\documentclass[a4paper,11pt,leqno]{article}
\usepackage{geometry}
\usepackage{layout}
\usepackage[french,english]{babel}
\usepackage[applemac]{inputenc}

\usepackage{mathtools}
\usepackage{array}
\usepackage{cite}
\usepackage{hyperref}
\usepackage{amsmath,amssymb,amsthm}
\usepackage[cyr]{aeguill}
\usepackage{esint}
\usepackage[usenames,dvipsnames]{color}
\usepackage{xargs}
\usepackage{enumitem}
\setlist{nosep} 

\newcommand{\N}{\mathbb{N}}

\newcommand{\R}{\mathbb{R}}
\newcommand{\C}{\mathbb{C}}

\newtheorem{theo}{Theorem}
\newtheorem{prop}{Proposition}[section]
\newtheorem{lem}[prop]{Lemma}
\newtheorem{coro}[prop]{Corollary}

\newtheorem{defi}[prop]{Definition}

\newtheorem{assumption}{Assumption}

\theoremstyle{plain}

\numberwithin{equation}{section}

\newcommand{\then}{\Longrightarrow} 
\def\t0{\rightarrow 0} 
\def\ti{\rightarrow \infty} 

\def\XXint#1#2#3{{\setbox0=\hbox{$#1{#2#3}{\int}$}
     \vcenter{\hbox{$#2#3$}}\kern-.5\wd0}}

\def \rm{\mathrm}

\def \hal{\frac{1}{2}}

\def \div{\mathrm{div} \, } 
\def \1{\mathbf{1}} 

\def \dist{\mathrm{dist}}

\def \nab{\nabla}

\def \mueq{\mu_{\mathrm{eq}}} 
\def \mupeq{\mu'_{\mathrm{eq}}} 
\def \meq{m_{\mathrm{eq}}} 
\def \mpeq{m'_{\mathrm{eq}}} 
\def \muN{\mu_{N}}

\def \carr{C} 
\def \Nn{\mathcal{N}}
\def \D{\mathcal{D}}
\def \um{\underline{m}} 
\def \om{\overline{m}} 
\def \mc{\mathcal}

\def \hN{h_N} 
\def \wN{w_N} 

\def \bW{\mathbb{W}}

\def \W{\mathbb{W}}
\def \Wc{\mathbb{W}_{m}}

\def \config{\mathcal{X}} 
\def \Lploc{L^p_{\mathrm{loc}}}

\def \probas{\mathcal{P}}
\def \Pelec{P^{\mathrm{elec}}} 

\def \P{\mathbb{P}} 

\def \PNbeta{\P^{\beta}_{N}} 
\def \ZNbeta{Z_{N, \beta}} 
\def \KNbeta{K_{N, \beta}} 

\def \fPNbeta{\Pgot^{\beta}_{N}} 
\def \fPNbeta{\mathfrak{P}_{N,\beta}}

\def \PNbetaxb{\mathfrak{P}^{z_0}_{N,\beta,\deltap}} 

\def \iN{i_N}
\def \iNxa{i^{z_0}_{N,\delta}}
\def \iNxb{i^{z_0}_{N,\deltap}}

\def \Leb{\mathbf{Leb}}
\def \Poisson{\mathbf{\Pi}}
\def\Esp{\mathbf{E}} 
\def \Ent{\mathrm{Ent}}   
\def \ERS{\mathsf{ent}} 

\def \B{\textbf{B}} 

\def \fbeta{\mathcal{F}_{\beta}} 
\def \fbetax{\mathcal{F}_{\beta}^{\meq(z_0)}}

\def \cds{2\pi} 

\def \Eloc{E^{\mathrm{loc}}}
\def \Hloc{H^{\mathrm{loc}}}
\def \C{\mathcal{C}}

\def \Elec{\mathsf{Elec}}
\def \Eleco{\Elec^{0}}

\def \XN{\vec{X}_N} 
\def \XpN{\vec{X}^{'}_N}

\def \Cxa{C^{z_0}_{\delta}}

\def \Nxa{\mathcal{N}^{z_0}_{\deltap}}

\def \M{M^{\eta}} 

\def \AintN{A_N^{\rm{int}}}
\def \Cint{\C^{\rm{int}}}
\def \Eint{E^{\rm{int}}}

\def \Etran{E^{\rm{ext}}}
\def \Ctran{\C^{\rm{ext}}}

\def \Ctran{\C^{\rm{tran}}}
\def \Etran{E^{\rm{tran}}}

\def \Atot{A^{\rm{tot}}}

\def \Ctot{\C^{\rm{tot}}}
\def \Etot{E^{\rm{tot}}}

\def \deltap{\delta_1}
\def \deltapp{\delta_2}
\def \deltatp{\delta_3}

\def \h {h}
\def \Chol{C_{\mu}}
\def \AtranN{A^{\rm{tran}}_{N}}
\def \Next{N^{\rm{tran}}}
\def \Nint{N^{\rm{gen}}}
\def \Nmid{N^{\rm{mid}}}

\def \Ein{E^{\rm{in}}}
\def \Hin{H^{\rm{in}}}
\def \Eou{E^{\rm{out}}}
\def \Nin{N^{\rm{in}}}
\def \Nbou{N^{\rm{bou}}}

\def \Xin{\vec{X}^{\rm{in}}}
\def \Elecin{\mathsf{Elec}^{\rm{in}}}
\def \Xou{\vec{X}^{\rm{out}}}
\def \Elecou{\mathsf{Elec}^{\rm{out}}}
\def \Fou{F^{\rm{out}}}
\def \Fin{F^{\rm{in}}}
\def \Nou{N^{\rm{out}}}
\def \Nin{N^{\rm{in}}}
\def \Ntran{N^{\rm{tran}}}

\def \cA{\mathcal{A}}
\def \cAtot{\cA^{\rm{tot}}}

\def \Zeta{\tilde{\zeta}}
\def \fbarbeta{\bar{\mathcal{F}}_{\beta}}
\def \Nnz{\Nn^{z_0}}

\def \M{\mathit{M}}
\def \Escr{E^{\rm{scr}}}
\def \Emod{E^{\rm{mod}}}

\def \KNbetax{K^{\beta}_{N,z, \deltap}}

\def \Cscr{\mathcal{C}^{\rm{scr}}}
\def \Cmod{\mathcal{C}^{\rm{mod}}}

\def \epsilon{\varepsilon}
\begin{document}
\title{Local microscopic behavior for $2$D Coulomb gases}
\author{Thomas Leblé\footnote{Sorbonne Universités, UPMC Univ. Paris 06 and CNRS, UMR 7598, Laboratoire Jacques-Louis Lions, F-75005, Paris. E-mail \texttt{leble@ann.jussieu.fr}} }
\maketitle
\begin{abstract}
The study of two-dimensional Coulomb gases lies at the interface of statistical physics and non-Hermitian random matrix theory. In this paper we give a large deviation principle (LDP) for the empirical fields obtained, under the canonical Gibbs measure, by zooming around a point in the bulk of the equilibrium measure, up to the finest averaging scale $N^{-1/2 + \epsilon}$. The rate function is given by the sum of the “renormalized energy” of Serfaty \textit{et al.} weighted by the inverse temperature, and of the specific relative entropy. We deduce a local law which quantifies the convergence of the empirical measures of the particles to the equilibrium measure, up to the finest scale. 
\end{abstract}
\section{Introduction}
\subsection{General setting}
We consider a system of $N$ points in the Euclidean space $\R^2$ with pairwise logarithmic interaction, in a confining potential $V$, and associate to any $N$-tuple $\XN = (x_1, \dots, x_N)$ the energy
\begin{equation} \label{HN}
\hN(\XN) := \sum_{1 \leq i \neq j \leq N} -\log|x_i-x_j| + N \sum_{i=1}^N V(x_i), \quad x_i \in \R^2.
\end{equation}
We only impose mild conditions on the potential $V$ (see Assumption \ref{assumption:V}). 

For any value of the \textit{inverse temperature} parameter $\beta > 0$ we consider the associated $N$-point Gibbs measure, which is absolutely continuous with respect to the Lebesgue measure on $(\R^2)^N$ with a density given by
\begin{equation}
\label{def:Gibbs} d\PNbeta(\XN) := \frac{1}{\ZNbeta} e^{-\hal \beta \hN(\XN)} d\XN,
\end{equation}
where we denote a $N$-tuple of points by $\XN = (x_1, \dots, x_N)$ and $d\XN := dx_1 \dots dx_N$.
The constant $\ZNbeta$ is a normalizing constant, also called the \textit{partition function}, so that the total mass of $\PNbeta$ is $1$.

\paragraph{Motivations.}
The model described by \eqref{HN} and \eqref{def:Gibbs} is known in statistical physics as a \textit{two-dimensional Coulomb gas}, \textit{two-dimensional log-gas} or \textit{two-dimensional one-component plasma}, we refer e.g. to \cite{AJ}, \cite{JLM}, \cite{SM} for a physical treatment of its main properties.

When $\beta = 2$ and $V$ is quadratic, the probability measure \eqref{def:Gibbs} coincides with the joint law of eigenvalues of a non-Hermitian matrix model known as the \textit{complex Ginibre ensemble}, which is obtained by sampling a $N \times N$ matrix whose coefficients are (properly normalized) i.i.d. complex Gaussians, see \cite{ginibre}. For $\beta = 2$, more general potentials can be considered, which are associated to “random normal matrices” (see e.g. \cite{ahm2}).  Systems of particles with a logarithmic interaction as in \eqref{HN}, called \textit{log-gases}, have been also (and mostly) been studied on the real line, motivated by their link with Hermitian random matrix theory. We refer to \cite{forrester} for a survey of the connection between log-gases and random matrix theory, and in particular to \cite[Chap.15]{forrester} for the two-dimensional (non-Hermitian) case.

The Ginibre case (and the case $\beta = 2$ in general) has the special property that the point process associated to $\PNbeta$ becomes \textit{determinantal}, which allows for an exact computation of many interesting quantities, e.g. the $n$-point correlation functions. The existence of a matrix model also allows for universality results at the microscopic scale as in \cite{Bourgade2d1, Bourgade2d2}. In the present paper we rather work with general $\beta > 0$ and potential $V$, thus dealing with what could be called \textit{two-dimensional $\beta$-ensembles} by analogy with the one-dimensional $\beta$-ensembles which generalize the laws of eigenvalues of random Hermitian matrices (see e.g. \cite{de}). The microscopic behavior of one-dimensional $\beta$-ensembles has been recently investigated in \cite{Bourgade1d1, Bourgade1d2} and we aim at a similar understanding in the two-dimensional case.

\paragraph{First-order results: the macroscopic behavior.}
Let us first recall some results about the macroscopic behavior of the particle system as $N \ti$.

If the potential $V$ has some regularity and grows fast enough at infinity (see Assumption \ref{assumption:V}) there is an associated \textit{equilibrium measure} $\mueq$, such that the sequence $\{\mu_N\}_N$ (where $\muN := \frac{1}{N} \sum_{i=1}^N \delta_{x_i}$ denotes the \textit{empirical measure} of the points) converges almost surely to $\mueq$. Moreover the law of $\{\mu_N\}_N$ satisfies a Large Deviation Principle (LDP) at speed $\frac{\beta}{2} N^2$ on the space $\probas(\R^2)$ of probability measures, with good rate function given by 
\begin{equation} \label{def:I}
I(\mu) := \iint -\log|x-y| d\mu(x) d\mu(y) + \int V(x) d\mu(x).
\end{equation}
This characterizes the \textit{first-order} or \textit{macroscopic} behavior of the interacting particle system. Typically, as $N$ becomes large, the $N$ points $x_1, \dots, x_N$ arrange themselves according to the probability density $d\mueq$, which has compact support $\Sigma$. Events that deviate from this prediction occur only with $\PNbeta$-probability of order $\exp(-N^2)$. We refer to \cite[Chap.2]{serfatyZur} and the references therein for a detailed exposition.

\paragraph{Microscopic behavior with macroscopic average.}
In this section we summarize the main result of \cite{LebSer}, which describes the behavior as $N \ti$ of a \textit{microscopic} quantity obtained through a \textit{macroscopic average}.

Let $\config$ be the set of locally finite point configurations in $\R^2$, endowed with the topology of vague convergence, and let us denote by $\probas(\config)$ the set of Borel probability measures on $\config$ i.e. the set of random point processes on $\R^2$ (we refer to Section \ref{sec:configpp} for more details).

In \cite{LebSer} (following the line of work \cite{gl13}, \cite{SS2d}, \cite{SS1d}, \cite{RougSer}, \cite{PetSer}) S. Serfaty and the author have investigated the microscopic behavior of the system by making a statement on the point processes arising when zooming in by a factor $N^{1/2}$ (which is the typical inter-particle distance) and averaging over translations in a way that we now briefly present.

For any $N$-tuple $\XN$, let
$x'_i = N^{1/2} x_i$, $\nu'_N := \sum_{i=1}^N \delta_{x'_i}$, let $\Sigma' := N^{1/2} \Sigma$ denote the support of $\mueq$ after rescaling, and let $\iN$ be the map $\iN : (\R^2)^N  \rightarrow  \probas(\config)$ defined by
\begin{equation} \label{def:iN}
\iN(\XN)  : =  \frac{1}{|\Sigma'|} \int_{\Sigma'} \delta_{\theta_{z'} \cdot \nu'_N}\, dz',
\end{equation}
where $\theta_{z'} \cdot$ denotes the action of translation by $z' \in \Sigma'$, and where $\delta$ is the Dirac mass. The map $\iN$ transforms a $N$-tuple of points into the data of all the blown-up point configurations obtained by zooming in by a factor $N^{1/2}$ around any $z \in \Sigma$. Such quantities are called \textit{empirical fields}.

We let $\fPNbeta$ be the push-forward of $\PNbeta$ by $\iN$. The main result of \cite{LebSer} gives a large deviation principle for $\{\fPNbeta\}_N$ at speed $N$, on the space of stationary random point processes. The rate function on this subset of $\probas(\config)$ is given by
\[
\fbeta(P) := \frac{\beta}{2} \Esp_{P} [\bW] + \ERS[P|\Poisson^1],
\]
where $\bW$ is an energy functional which will be defined later, $\Esp_{P}$ denotes the expectation under $P$, and $\ERS[P|\Poisson^1]$ is the \textit{specific relative entropy} of $P$ with respect to the Poisson point process of intensity $1$ in $\R^2$ (see Section \ref{sec:entropy}).

This LDP characterizes the microscopic behavior only in an \textit{averaged} way, because of the average over translations in the definition of $\iN$. In fact (this is still a consequence of \cite[Theorem 1]{LebSer}) this description can be enhanced by replacing the average over translations in $\Sigma'$ by an average over translations in arbitrary small macroscopic regions (seen in blown-up scale), for example the square $C(z'_0, \epsilon N^{1/2})$, where $\epsilon > 0$ is fixed and $z'_0 = N^{1/2}z_0$ for some $z_0$ in the interior of $\Sigma$ (the bulk). Let us emphasize that the average still takes place at the macroscopic scale $N^{1/2}$. This is done in \cite{LebSer} by considering “tagged” empirical fields  which are elements of $\probas(\Sigma \times \config)$ keeping track of the point around which the configuration has been zoomed, thus allowing for a macroscopic localization.

\paragraph{Microscopic behavior with mesoscopic averages.}
The goal of this paper is to push further the analysis of \cite{LebSer} at finer scales and to consider \textit{mesoscopic} versions of the map $\iN$. In other terms we look at the empirical fields obtained by averaging over translations in $C(z_0', N^{\delta})$ for $0 < \delta < 1/2$, and we obtain a LDP at speed $N^{2\delta}$ with essentially the same rate function as above. It is crucial to average over a relatively large set and although one might hope to go down to even finer scales (e.g. $O(\log^k N)$ for $k$ large enough) we do not expect a similar result to hold for a strictly speaking \textit{microscopic} average at scale $O(1)$ (in blown-up coordinates). 

The first-order results show that the empirical measure $\mu_N := \frac{1}{N} \sum_{i=1}^N \delta_{x_i}$ converges to the equilibrium measure $\mueq$ almost surely. As a consequence of our analysis we get a “local law” (borrowing the terminology of \cite{Bourgade2d1} and \cite[Theorem 20]{TaoVu}) which implies that $\mu_N$ and $\mueq$ are close at small scales with very high probability.

\subsection{Preliminary notation and definitions}
\subsubsection{Notations}
For $R > 0$ we denote by $C_R$ the square $[-R/2, R/2]^2$ and by $C(z,R)$ the translate of $C_R$ by $z \in \R^2$. We denote by $D(p,r)$ the disk of center $p$ and radius $r > 0$. If $N$ is fixed and $\XN \in (\R^2)^N$ we denote by $\nu_N := \sum_{i=1}^N \delta_{x_i}$ and $\nu'_N := \sum_{i=1}^N \delta_{x'_i}$ (where $x'_i = N^{1/2} x_i$).

Let $0 < \delta < 1/2$. We say that an event $\mathcal{A}$ occurs with $\delta$-overwhelming probability if 
\[
\limsup_{N \ti} N^{-2\delta} \log \PNbeta(\mathcal{A}^c) = -\infty,
\]
where $\mathcal{A}^c$ is the complement of $\mathcal{A}$. In particular, for any event $\mc{B}$, if $\mathcal{A}$ occurs with $\delta$-overhelming probability we have
\[
\limsup_{N \ti} N^{-2\delta} \log \PNbeta(\mc{B}) = \limsup_{N \ti}  N^{-2\delta} \log \PNbeta(\mc{B} \cap \mc{A}),
\]
and the same goes for the $\liminf$. In other terms, when evaluating probabilities of (logarithmic) order $N^{2\delta}$ we may restrict ourselves to the intersection with any event of $\delta$-overhelming probability.

If $\{a_N\}_N, \{b_N\}_N$ are two sequences of non-negative real numbers, we will write $a_N \preceq b_N$ if  if there exists $C > 0$ such that $a_N \leq C b_N$ ($\PNbeta$-a.s. if the numbers are random), and we will write $a_N \preceq_{\delta} b_N$ if there exists $C > 0$ such that $a_N \leq C b_N$ with $\delta$-overhelming probability.

We will write
$a_N \ll N^{\delta}$ if there exists $\tau > 0$ such that $a_N \leq N^{\delta-\tau}$ ($\PNbeta$-a.s. if the numbers are random) and $a_N \ll_{\delta'} N^{\delta}$ if there exists $\tau > 0$ such that $a_N \leq N^{\delta-\tau}$ with $\delta'$-overwhelming probability.

\subsubsection{Equilibrium measure and splitting of the energy}
Under mild hypotheses on $V$ (see Assumption \ref{assumption:V}) it is known (see e.g. \cite[Chap.1]{safftotik}) that there exists a probability measure $\mueq$ with compact support $\Sigma$ which is the unique minimizer of $I$ (as in \eqref{def:I}) over $\probas(\R^2)$ (the set of probability measures). Defining $\zeta$ as
\begin{equation}
\zeta(x) := \int_{\R^2} - \log |x-y| d\mueq(x) + \frac{V}{2} - \left(\iint_{\R^2} - \log |x-y| d\mueq(x) d\mueq(y) - \frac{1}{2} \int_{\R^2} V d\mueq\right),
\end{equation}
we have $\zeta \geq 0$ quasi-everywhere (q.e.) in $\R^2$ and $\zeta = 0$ q.e. on $\Sigma$, and in fact this characterizes $\mueq$ uniquely, see \cite{frostman}. If $N \geq 1$ is fixed we let $\mupeq(x) := \mueq(xN^{-1/2})$.

If $\C$ is a finite point configuration we define the second-order energy functional $\wN(\C)$ as
\begin{equation}
\label{def:wN}
\wN(\C) := \iint_{\triangle^c} - \log |x-y| (d\C - d\mupeq)(x) (d\C - d\mupeq)(y),
\end{equation}
where $\triangle^c$ denotes the complement of the diagonal $\triangle$. It computes the electrostatic interaction of the electric system made of the point charges in $\C$ and a negatively charged background of density $\mupeq$, without the infinite self-interactions of the point charges.

Let $\Zeta(\C) : = \int \zeta d\C$. It was proven in \cite{SS2d} (see also \cite[Chap.3]{serfatyZur}) that the following \textit{exact splitting formula} holds:
\begin{lem} \label{lem:splitting} For any $N \geq 1$ and any $\XN \in (\R^2)^N$ we have, with $I$ as in \eqref{def:I}
\begin{equation} \label{splitting}
 \hN(\XN) = N^2 I(\mueq) -\frac{N \log N}{2} + \wN(\nu'_N) + 2N \Zeta(\nu_N).
\end{equation}
\end{lem}

We may thus re-write the Gibbs measure $\PNbeta$ as 
\begin{equation} \label{def:Gibbs2}
d\PNbeta(\XN) = \frac{1}{\KNbeta} e^{-\frac{1}{2} \beta(\wN(\nu'_N) + 2N \Zeta(\nu_N))} d\XN,
\end{equation}
where $\KNbeta$ is a new normalizing constant. The exponent $(\wN(\nu'_N) + 2N \Zeta(\nu_N))$ is expected to be typically of order $N$, and it was proven in \cite[Cor. 1.5]{LebSer} that $\log \KNbeta = - N \min \fbarbeta + o(N)$, where $\fbarbeta$ is closely related to the function $\fbeta$ mentioned above.

\subsubsection{Energy and entropy}
\paragraph{Renormalized energy.}
In \cite{LebSer}, following \cite{gl13, SS2d, RougSer, PetSer}, an energy functional is defined at the level of random stationary point processes (see also \cite[Chap.3-6]{serfatyZur}), which is the $\Gamma$-limit of $\frac{1}{N} \wN$ as $N \ti$. We will define it precisely in Section \ref{sec:energie} and we denote it by $\Wc$ (where $m \geq 0$ is a parameter - the notation differs slightly from that of \cite{LebSer} where it corresponds to $\widetilde{\mathbb{W}}_m$). It can be thought of as the infinite-volume limit of \eqref{def:wN} and as a way of computing the interaction energy of an infinite configuration of point charges $\C$ together with a negatively charged background of constant density $m$.

\paragraph{Specific relative entropy.}
For any $m \geq 0$ we let $\Poisson^m$ be the law of a Poisson point process of intensity $m$ in $\R^2$. Let $P$ be a stationary random point process on $\R^2$. The relative specific entropy $\ERS[P|\Poisson^m]$ of $P$ with respect to $\Poisson^m$ is defined by
\begin{equation} \label{def:ERS}
\ERS[P|\Poisson^m] := \lim_{R \ti} R^{-2} \Ent\left(P_{|\carr_R} | \Poisson^m_{|\carr_R} \right),
\end{equation}
where $P_{|\carr_R}$ denotes the random point process induced in $\carr_R$, and $\Ent( \cdot | \cdot)$ denotes the usual relative entropy (or Kullbak-Leibler divergence) of two probability measures defined on the same probability space. We take the appropriate sign convention for the entropy so that it is non-negative: if $\mu,\nu$ are two probability measures defined on the same space  we let $\Ent \left(\mu | \nu \right) := \int \log \frac{d\mu}{d\nu} d\mu$ if $\mu$ is absolutely continuous with respect to $\nu$ and $+ \infty$ otherwise. For more details we refer to Section \ref{sec:entropy}.

\subsubsection{Good control on the energy} \label{sec:goodcont}
In this paragraph we define the notion of “good control at scale $\delta$”, which expresses the fact that our particle system has good properties in any square of sidelength $N^{\delta}$ (after blow-up). The assumption that good control at scale $\delta$ holds will be a key point in order to prove the LDP at  slightly smaller scales. Moreover we will see that the “good control” assumption can be bootstrapped, i.e. good control at scale $\delta$ implies good control at scale $\deltap$ for $\deltap < \delta$ large enough.

Let us first introduce the local electric field $\Eloc$ and its truncation $\Eloc_{\eta}$, we will come back to these definitions in more detail in Section \ref{sec:Eloc}. If $\XN$ is a $N$-tuple of points in $\R^2$, for any $0 < \eta < 1$ we denote by $\nu'_{N,\eta}$ the measure $\nu'_{N, \eta} := \sum_{i=1}^N \delta^{(\eta)}_{x'_i}$,
where $\delta^{(\eta)}_{x'_i}$ denotes the uniform probability measure on the circle of center $x'_i$ and radius $\eta$.
We let $\Eloc$ be the associated “local electric field” $\Eloc(x) := (- \nabla \log) * (\nu'_N - \mupeq)$ and $\Eloc_{\eta}$ its truncation at scale $\eta$, defined by $
\Eloc_{\eta}(x) := (- \nabla \log) * (\nu'_{N,\eta} - \mupeq)$. Finally, we denote by $\mathring{\Sigma}$ the interior of $\Sigma$. 

\begin{defi}  \label{def:goodcontrol}
For any $0 < \delta \leq \frac{1}{2}$ we say that a good control at scale $\delta$ holds if for any $z_0 \in \mathring{\Sigma}$ and any $0 < \deltap < \delta$ we have with $\deltap$-overwhelming probability:
\begin{enumerate}
\item The number of (blown-up) points $\mathcal{N}^{z_0}_{\delta}$ in the square $C(z'_0,N^{\delta})$ is of order $N^{2\delta}$
\begin{equation} \label{goodcont2}
\mathcal{N}^{z_0}_{\delta} \preceq_{\deltap} N^{2\delta}.
\end{equation}
\item For any $0 < \eta < 1$ we have
\begin{equation}\label{goodcont1}
\int_{\Cxa} |\Eloc_{\eta}|^2 + \mathcal{N}^{z_0}_{\delta} \log \eta  \preceq_{\deltap} N^{2\delta},
\end{equation}
which expresses the fact that the energy in the square $C(z'_0,N^{\delta})$ (after blow-up) is of order $N^{2\delta}$.
\end{enumerate}
\end{defi}
Let us emphasize that \eqref{goodcont2}, \eqref{goodcont1} control quantities at scale $\delta$ by looking at probabilities at scale $\deltap < \delta$.

\subsection{Rate function}
Let us define the \textit{local} rate function as
\begin{equation}
\label{def:fbetax}
\fbeta^m(P) := \frac{\beta}{2} \bW_{m}(P) + \ERS[P | \Poisson^{m}].
\end{equation}
It is a \textit{good rate function} because both terms are good rate functions (see e.g. \cite[Lemma 4.1]{LebSer}).

For any $m > 0$ we let $\probas_{s,m}(\config)$ be the set of random stationary point processes of intensity $m$. Let us define a scaling map $\sigma_m : \probas_{s}(\config) \to \probas_{s}(\config)$ such that $\sigma_m(P)$ is the push-forward of $P$ by $\C \mapsto m^{-1/2} \C$. It is easy to see that $\sigma_m$ induces a bijection from $\probas_{s,m}(\config)$ to $\probas_{s,1}(\config)$ for any $m > 0$. 

It is proven (see \cite[Def. 2.4, Lemma 4.2]{LebSer}) that
\begin{lem} \label{lem:scaling} The map $\sigma_m$ induces a bijection between the minimizers of $\fbeta^m$ over $\probas_{s,m}(\config)$ and the minimizers of $\fbeta^1$ over $\probas_{s,1}(\config)$.
\end{lem}

We may now state our main results. 
\subsection{Statement of the results}
If $z_0 \in \mathring{\Sigma}$ and $0 < \deltap < 1/2$ are fixed, let us define the map $\iNxb : \config \rightarrow \probas(\config)$ as
\begin{equation} \label{def:iNxa}
\iNxb(\C) :=  N^{-2 \deltap} \int_{C(z_0', N^{\deltap})} \delta_{\theta_z' \cdot (\C \cap C(z_0', N^{\deltap}))} dz'.
\end{equation}
Such quantities are called empirical fields. We denote by $\PNbetaxb$ the law of the push-forward of $\PNbeta$ by the map $\iNxb$ - in other words: the empirical field observed around $z_0$ by averaging at the mesoscopic scale $N^{\deltap}$. Finally we denote by $\meq$ the density of $\mueq$ (see Assumption \ref{assumption:V}).

\begin{theo} \label{theo:main}
For any $0 < \deltap < 1/2$, for any $z_0$ in $\mathring{\Sigma}$, the sequence $\{\PNbetaxb\}_N$ obeys a large deviation principle at speed $N^{2\deltap}$ with good rate function $\left(\fbeta^{\meq(z_0)} - \min \fbeta^{\meq(z_0)} \right)$.

Moreover a good control at scale $\deltap$ holds in the sense of Definition \ref{def:goodcontrol}.
\end{theo}
Theorem \ref{theo:main} tells us in particular that the behavior around $z_0 \in \mathring{\Sigma}$ depends on $V$ only through the value $\meq(z_0)$, and in view of Lemma \ref{lem:scaling} it has only the effect of scaling the configurations. This yields another example of the \textit{universality} phenomenon: the small scale behavior of the particle system is essentially independent of the choice of $V$.

The first consequence of Theorem \ref{theo:main} is a bound on the discrepancy i.e. the difference between the number of points of $\nu'_N$ in a given square and the mass given by $\mupeq$.
\begin{coro} \label{coro:discr}
Let $z_0 \in \mathring{\Sigma}$, let $0 < \delta < 1/2$ and $\deltap \in (\delta/2, \delta)$. We have
\begin{equation} \label{discr}
\left| \int_{C(z'_0, N^\deltap)}  d\nu'_N - d\mupeq \right| \ll_{\delta_1} N^{\frac{4}{3} \delta}.
\end{equation}
\end{coro}
For $\deltap < \delta$ close to $\delta$, the bound $N^{\frac{4}{3}\delta}$ on the difference is much smaller than the typical value of each term, of order $N^{2\deltap}$. It allows us to prove a \textit{local law} in the following sense:
\begin{coro} \label{coro:locallaw}
Let $z_0 \in \mathring{\Sigma}$ and $0 < \delta \leq 1/2$ be fixed. Let $f$ be a $C^1$ function (which may depend on $N$) such that $f$ is supported in $C(z'_0, N^{\delta})$. Then for any $\delta/2 < \deltap < \delta$ we have
\begin{equation} \label{localaw}
N^{-2\delta} \left| \int_{C(z'_0, N^{\delta})} f (d\nu'_n - d\mupeq) \right| \preceq_{\deltap} \|\nabla f\|_{\infty} N^{\deltap} + \|f\|_{\infty} N^{-2\delta/3}.
\end{equation}
In particular if $f(z) = \tilde{f}(N^{-\delta}(z-z'_0))$ for some compactly supported $C^1$ function $\tilde{f}$ then $\|\nabla f\|_{\infty} \preceq N^{-\delta}$ and $\|f\|_{\infty} \preceq 1$, thus we get
\[
N^{-2\delta} \left| \int_{C(z'_0, N^{\delta})} f (d\nu'_n - d\mupeq) \right| \ll_{\deltap} 1.
\]
\end{coro}

\paragraph{Comments and open questions.}
In the statement of the results we restrict ourselves to the following setting: we first pick a point $z_0$ in the interior of $\Sigma$ (called the \textit{bulk}) and then look at the point process in $C(z', N^{\delta})$ with $N^{-1/2} z' = z_0$. A careful inspection of the proof shows that we might have taken $z'$ depending on $N$ more finely, e.g. by considering a sequence $z'$ with $N^{-1/2} z' \to z_0 \in \mathring{\Sigma}$, while keeping the same conclusions. It does not seem possible to take $z' \to z_0 \in \partial \Sigma$ (the “edge case”) in general because the density $\meq$ may vanish near the boundary - however, this does not happen in the standard example of the quadratic potential, in which the density is constant up to the boundary of the support. Our analysis might be done in the edge case at a scale $\delta \geq \delta_c$ depending on the speed at which $\meq(z)$ vanishes, but we do not pursue this goal here.

The minimizers of the rate function are unknown in general, however it is proven in \cite[Corollary 1.4]{LebSer} that the Ginibre point process minimizes $\fbeta^{1}$ over $\probas_{s, 1}(\config)$ for $\beta = 2$.  We do not know whether uniqueness of the minimizers holds for $\beta =2$, nor for any value of $\beta$. Uniqueness of the minimizers for some $\beta > 0$ would imply that the empirical fields have a limit in law as $N \ti$, which would heuristically correspond to some “$\beta$-Ginibre” random point process. In that case, our results shows that the hypothetical convergence 
\[
\text{Empirical field averaged at scale $N^{\delta}$} \rightarrow \text{$\beta$-Ginibre}
\]
holds at arbitrarily fine scales $\delta > 0$, which would hint at the convergence in law of the non-averaged point process $\nu'_N$ to the conjectural $\beta$-Ginibre point process.

Another open question is the behavior of the minimizers as $\beta \ti$ (the low-temperature limit). The \textit{crystallization conjecture} (see e.g. \cite{crystal} for a review) predicts that the minimum of $\bW_{1}$ on $\probas_{s,1}(\config)$ is (uniquely) attained by the random stationary point process associated to the triangular lattice. In the high-temperature limit, it is proven in \cite[Theorem 2]{WBS} that minimizers of $\fbeta^1$ converge (in a strong sense) to $\Poisson^1$ as $\beta \t0$.

The result of \cite{LebSer} and most of the methods used in this paper are valid in a broader setting than the two-dimensional, logarithmic case, in particular we could think of treating the $1d$ log-gas (i.e. the $\beta$-ensembles). It turns out that an adaptation of the present method in the one-dimensional case allows one to improve the result of \cite{LebSer} to finer, mesoscopic scales, however, we have been unable so far to go down to the finest scale $N^{-1 + \epsilon}$ and we hope to return on this question in a subsequent work.

\subsection{Plan of the paper and sketch of the proof}
In Section \ref{sec:notations} we introduce some notation and we give the definitions of the main objects used throughout the paper, as well as their key properties. In Section \ref{sec:preliminaries} we gather preliminary results about the energy $\wN$ and we prove the main technical tool, called the “screening lemma”. Section~\ref{sec:LDPUB} is devoted to the proof of a LDP upper bound and Section~\ref{sec:LDPUB} to the lower bound. We combine these two steps to prove Theorem~\ref{theo:main} in Section~\ref{sec:conclusion}, together with Corollary~\ref{coro:discr} and \ref{coro:locallaw}. Section~\ref{sec:annexe} is devoted to intermediate results which we postpone there.

Let us now sketch how the proof of Theorem \ref{theo:main} goes. The basic idea is a bootstrap argument, we find that there exists $t < 1$ such that
\begin{equation} \label{bootstrap}
\text{Good control at scale $\delta$} \longrightarrow \left\lbrace \begin{array}{c}
\text{Large deviations at scale $\deltap$} \\
\text{Good control at scale $\deltap$}
\end{array}\right. \text{ for all $t\delta \leq \deltap < \delta$.}
\end{equation}
Once good control at scale $\delta = 1/2$ is established, Theorem \ref{theo:main} follows. A similar bootstrap argument was used in \cite{Rot-NodSer} for studying the minimizers of $\wN$ (which corresponds to the $\beta = + \infty$, or zero temperature case).

The main obstruction to obtaining LDP for empirical fields in our context is the non-locality of the energy \eqref{def:wN}: due to the long-range nature of the interactions, it is hard to localize the energy in a given square in such a way that it only depends of the point configuration in this square. Another way of seeing it is that $\Eloc(x)$ depends \textit{a priori} on the whole configuration $\XN$ and not only on the points close to $x$. 

To prove \eqref{bootstrap} we rely on the following steps: let $z_0 \in \mathring{\Sigma}$ be fixed. For the sake of simplicity let us assume that $\wN(\XN) = \frac{1}{2\pi} \int_{\R^2} |\Eloc|^2$, where $\Eloc$ is the local electric field defined in Section \ref{sec:goodcont} (see also Section \ref{sec:Eloc}).
\begin{enumerate}
\item For any $\XN$, we split the energy $\wN(\XN)$ as 
\[
\int_{\R^2} |\Eloc|^2 = \int_{C(z'_0, N^{\deltap})} |\Eloc|^2 + \int_{C(z'_0, N^{\deltap})^c} |\Eloc|^2,
\]
and we split $\XN$ as $\Xin + \Xou$ where $\Xin$ is the point configuration in $C(z'_0, N^{\deltap})$ and $\Xou$ is the point configuration in $C(z'_0, N^{\deltap})^c$.
\item We define $\Fin(\Xin)$ (resp. $\Fou(\Xou)$) as the minimal energy of an electric field associated to $\Xin$ (resp. $\Xou$). We thus have
\[
\wN(\XN) \geq \Fin(\Xin) + \Fou(\Xou).
\]
The two terms in the right-hand side become independent (they depend from two distinct sets of variables).
\item Inserting the previous inequality into the expression of the Gibbs measure \eqref{def:Gibbs2} we obtain, for any event $\mc{A}$ “concerning” $\Xin$ 
\begin{equation} \label{sketchdecompo}
\PNbeta(\mc{A}) \leq \frac{1}{\KNbeta} \left(\int_{\mc{A}} e^{- \frac{1}{2} \beta \Fin(\Xin)} d\Xin \right) \left(\int e^{- \frac{1}{2} \beta (\Fou(\Xou) + 2N\Zeta(\Xou))} d\Xou\right). 
\end{equation}
This can be used to prove a first LDP upper bound (taking $\mc{A} = \{\iNxb(\Xin) \in B(P,\epsilon)\}$) or a first “good control” estimate (taking $\mc{A} = \{\Fin(\Xin) \gg N^{2\deltap}\}$). 
\item Then we need to prove that \eqref{sketchdecompo} is sharp (at scale $\deltap$). Given $\Xou$ and $\Xin$, it amounts to be able to reconstruct (a family of) point configurations $\XN \approx \Xou + \Xin$ such that $\wN(\XN) \leq \Fin(\Xin) + \Fou(\Xou) + o(N^{2 \deltap})$. This is where the screening procedure is used: we modify $\Xou$ and the associated electric field a little bit (this procedure follows the line of work \cite{SS2d, SS1d, RougSer, PetSer} and is called \textit{screening} for reasons that will appear later) so that we may glue together $\Xin$ and the new $\Xou$ and create an electric field compatible with the new (slightly modified) point configuration $\XN$. It is then a general fact that $\wN(\XN)$ (the energy of the \textit{local} electric field associated to $\XN$) is the smallest energy in a wide class of compatible electric fields.

In particular, proving a partial converse to \eqref{sketchdecompo} allows us to estimate the “local partition function”
\[
\frac{1}{\KNbeta}\left(\int e^{- \frac{1}{2} \beta (\Fou(\Xou) + 2N\Zeta(\Xou))} d\Xou\right),
\] 
and also to show a LDP lower bound. Combined with the estimates of the previous step, it proves \eqref{bootstrap}.
\end{enumerate}

\paragraph{\textit{Acknowledgements.}} The author would like to thank his PhD supervisor, Sylvia Serfaty, for helpful comments on this work.

\section{Notations, assumptions and main definitions} \label{sec:notations}

\subsection{Assumption on the potential}
\begin{assumption} \label{assumption:V}
The potential $V$ is such that 
\begin{enumerate}
\item $V$ is lower semi-continuous (l.s.c.) and bounded below.
\item The set $\{x \in \R^2 \ | \ V(x) < \infty \}$ has positive logarithmic capacity.
\item We have  $\lim_{|x|\to \infty} \frac{V(x)}{2}- \log |x|= + \infty$.
\end{enumerate}
These first three conditions ensure that the equilibrium measure $\mueq$ is well-defined and has compact support $\Sigma$. 
Furthermore we ask that the measure $\mueq$ has a density $\meq$ which is $\kappa$-Hölder in $\Sigma$, for some $0 < \kappa \leq 1$
\begin{equation} \label{assum-Holder}
|\meq(x) - \meq(y)| \leq |x-y|^{\kappa}.
\end{equation}
\end{assumption}
If $V$ is $C^2$, it is known that $\mueq$ is absolutely continous with respect to the Lebesgue measure on $\R^2$ and its density coincides with $\Delta V$ in $\Sigma$. Thus in particular \eqref{assum-Holder} is satisfied as soon as $V$ is $C^{2,\kappa}$. Let us observe that the third assumption (that $V$ is “strongly confining”) could be slightly relaxed into 
\[
\liminf_{|x| \ti} \frac{V(x)}{2} - \log |x| > - \infty,
\]
i.e. $V$ is only “weakly confining”, in which case the support $\Sigma$ might not be compact (see \cite{Hardy} for a proof of the first-order LDP in this case). We believe that Theorem \ref{theo:main} should extend to the non-compact case as well, since it is really a \textit{local} result, but we do not pursue this goal here.

\subsection{Point configurations and point processes} \label{sec:configpp}
\paragraph{Point configurations.}
If $B$ is a Borel set of $\R^2$ we denote by $\config(B)$ the set of locally finite point configurations in $B$ or equivalently the set of non-negative, purely atomic Radon measures on $B$ giving an integer mass to singletons. We will often write $\mathcal{C}$ for $\sum_{p \in \mathcal{C}} \delta_p$. We endow the set $\config := \config(\R^2)$ (and the sets $\config(B)$ for $B$ Borel) with the topology induced by the topology of weak convergence of Radon measure (also known as vague convergence or convergence against compactly supported continuous functions), these topologies are metrizable and we fix a compatible distance $d_{\config}$. 

\paragraph{Volume of configurations.}
Let $B$ be a Borel set of $\R^2$. For any $N \geq 1$, let $\sim_N$ be the equivalence relation on $B^N$ defined as $(x_1, \dots, x_N) \sim_N (y_1, \dots, y_N)$ if and only if there exists a permutation $\sigma \in \mathfrak{S}_N$ (the symmetric group on $N$ elements) such that $x_i = y_{\sigma(i)}$ for $i =1, \dots, N$. We denote by $B^N / \mathfrak{S}_N$ the quotient set and by $\pi_N$ the canonical projection $B^N \to B^N / \mathfrak{S}_N$. The set of finite point configurations in $B$ can be identified to $
\{\emptyset\} \cup \bigcup_{N = 1}^{+\infty} B^N / \mathfrak{S}_N$.

If $\cA \subset B^N / \mathfrak{S}_N$ we define $\hat{\cA} \subset B^N$ as $\hat{\cA} := \bigcup_{\C \in \cA} \C.$ It is easy to see that $\hat{\cA}$ is the largest subset of $B^N$ such that the (direct) image of $\hat{\cA}$ by $\pi_N$ is $\cA$. 

We will call “the volume of $\cA$” and write (with a slight abuse of notation) $\Leb^{\otimes N}(\cA)$ the quantity $\Leb^{\otimes N}(\hat{\cA})$.

\paragraph{Random point process.}
A random point process is a probability measure on $\config$. We denote by $\probas_s(\config)$ the set of stationary random point processes i.e. those which are invariant under (push-forward by) the natural action of $\R^2$ on $\config$ by translations. We endow $\probas_s(\config)$ with the topology of weak convergence of probability measures, and we fix a compatible distance $d_{\probas(\config)}$, e.g. the $1$-Wasserstein distance. Throughout the text we will denote by $B(P, \epsilon)$ the closed ball of center $P$ and radius $\epsilon$ for $d_{\probas(\config)}$.

\subsection{Electric systems and electric fields} \label{sec:Eloc}
\paragraph{Finite electric system.}
We will call an “electric system” a couple $(\C, \mu)$ where $\C$ is a point configuration and $\mu$ is a non-negative measurable bounded function in $\R^2$. We say that the system is finite if $\C$ is finite and $\mu$ is compactly supported. We say that the system is neutral if it is finite and $\int_{\R^2} d\C = \int_{\R^2} \mu(x) dx$.

\paragraph{Electric fields.} Let $1 < p < 2$ be fixed. We define the set of electric fields $\Elec$ as the set of vector fields in $\Lploc(\R^2, \R^2)$ such that
\begin{equation}
\label{def:Elec} - \div E = \cds \left( \C - \mu \right) \text{ in } \R^2
\end{equation}
for some electric system $(\C, \mu)$. When \eqref{def:Elec} holds we say that $E$ is compatible with $(\C, \mu)$ in $\R^2$ and we denote it by $E \in \Elec(\C, \mu)$. 
If $K$ is a compact subset of $\R^2$ with piecewise $C^1$ boundary we let $\Elec(\C,\mu, K)$ be the set of electric fields which are compatible with $(\C, \mu)$ in $K$ i.e. such that 
\begin{equation*}
\label{def:Elec2} - \div E = \cds \left( \C - \mu \right) \text{ in } K.
\end{equation*}

We denote by $\Eleco$ the set of decaying electric fields, such that $E(z) = O(|z|^{-2})$ as $|z| \ti$. We let $\Eleco(\C, \mu,K)$ be the set of electric fields which are compatible with $\C, \mu$ in $K$ \textit{and} decay.

\paragraph{Local electric fields.} 
If $(\C, \mu)$ is a finite electric system there is a natural compatible electric field, namely the local electric field defined as $\Eloc :=  - \nabla \log * (\C - \mu)$. We also define the “local electric potential” $\Hloc := - \log * (\C - \mu )$.
The scalar field $\Hloc$ corresponds physically to the electrostatic potential generated by the point charges of $\C$ together with a  background of “density” $\mu$. The vector field $\Eloc$ can be thought of as the associated electrostatic field. It is easy to see that $\Eloc$ fails to be in $L^2_{\rm{loc}}$ because it blow ups like $|x|^{-1}$ near each point of $\C$, however $\Eloc$ is in $\Lploc(\R^{2}, \R^{2})$ for any $1 < p < 2$.

\paragraph{Truncation procedure.}
The renormalization procedure of \cite{RougSer,PetSer} uses a truncation of the singularities which we now recall. We define the truncated Coulomb kernel as follows:
for  $0 < \eta < 1$ and  $x \in \R^2$, let $f_\eta(x)= \left(- \log |x|- \log \eta\right)_+$. If $(\C, \mu)$ is an electric system and $E \in \Elec(C, \mu)$ we let
\begin{equation}
\label{def:Eeta} E_{\eta}(X) := E(X) - \sum_{p \in \C} \nabla f_\eta (X -p).
\end{equation}

\subsection{Renormalized energy} \label{sec:energie}
\paragraph{For finite point configurations.}
It follows from the definition, and the fact that $-\log$ is (up to a constant) the Coulomb kernel in dimension $2$, that  $- \Delta \Hloc = \cds (\C - \mu)$, where $\Hloc$ denotes the local electric potential associated to a finite electric system $(\C, \mu)$. We may thus observe that $\wN(\C)$ (defined in \eqref{def:wN}) can be written
\[\wN(\C) \approx - \frac{1}{2\pi} \int \Hloc \Delta \Hloc \] (up to diagonal terms).
Using $\Eloc = \nabla \Hloc$ and integrating by parts we obtain heuristically $\wN(\C) = \int_{\R^2} |\Eloc|^2$. However, this computation does not make sense because $\Eloc$ fails to be in $L^2$ around each point charge (and indeed the diagonal is excluded in the definition of $\wN$). Following the method of \cite{bbh}, the correct way of giving a sense to “$\wN(\C) \approx \int_{\R^2} |\Eloc|^2$” is to use a \textit{renormalization} procedure, using the truncation at scale $\eta$ defined above. The following is proven in \cite{SS2d}.
\begin{lem} \label{lem:REfinite} For any $N \geq 1$ and any $\XN \in (\R^2)^N$, we have
\begin{equation} \label{REfinite}
\wN(\nu'_N) = \frac{1}{\cds}  \lim_{\eta \t0} \int_{\R^2} \left(|\Eloc_{\eta}|^2 + N \log \eta \right).
\end{equation} 
Moreover $\wN$ is bounded below on $(\R^2)^N$ by $O(N)$.
\end{lem}

\paragraph{For infinite electric fields.}
Let $(\C, \mu)$ be an electric system and $E \in \Elec(\C, \mu)$. 
We let $\mc{W}_{\eta}(E)$ be
\[
\mc{W}_{\eta}(E) := \frac{1}{2\pi} \limsup_{R \ti} R^{-2}  \int_{C_R} \left(|E_{\eta}|^2 +  \mu \log \eta\right).
\]
The renormalized energy of $E$ is then defined as $\mc{W}(E) := \limsup_{\eta\to 0} \mc{W}_{\eta}(E)$.

\paragraph{For (random) infinite point configurations.} 
If $(\C, \mu)$ is an electric system with $\mu$ constant equal to some $m > 0$ we define 
\begin{equation} \label{def:WdeE}
\Wc(\C) = \inf_{E \in \Elec(\C, \mu)} \mc{W}(E).
\end{equation}
Similarly if $P$ is a random point process we let $\Wc(P) = \Esp_{P} \left[ \Wc \right]$ for any $m > 0$.

The following lower semi-continuity result was proven in \cite[Lemma 4.1]{LebSer}.
\begin{lem} \label{lem:lsciW}
For any $m > 0$ the map $P \mapsto \Wc(P)$ is lower semi-continuous on $\probas_{s}(\config)$. Moreover its sub-level sets are compact. In particular, $\Wc$ is bounded below.
\end{lem}

\subsection{Specific relative entropy} \label{sec:entropy}
For any $P \in \probas_{s}(\config)$, the specific relative entropy $\ERS[P|\Poisson^m]$ is defined as in \eqref{def:ERS} (see e.g. \cite{Georgii1}).
\begin{lem} \label{lem:ERS}
For any $m \geq 0$ the map $P \mapsto \ERS[P|\Poisson^m]$ is well-defined on $\probas_{s}(\config)$, it is affine lower semi-continous and its sub-level sets are compact.
\end{lem}
\begin{proof}
 We refer to \cite[Chap. 6]{seppalainen} for a proof of these statements.
\end{proof}

\paragraph{Large deviations for the reference measure.}
For $\delta  >0$, $N \geq 1$ and $S_N \in \N$, let $\textbf{B}_{S_N,N,\delta}$ be the law of the Bernoulli point process with $S_N$ points in $C(0,N^{\delta})$. 
\begin{prop} \label{prop:Sanov}
Let $P \in \probas_{s,c}(\config)$ and $\delta > 0$, let $\{S_N\}_N$ be such that $S_N \sim_{N \ti} m N^{2\delta}$ for some $m > 0$. We have
\begin{equation} \label{Sanovc}
\lim_{\epsilon \t0} \lim_{N \ti} N^{-2\delta} \log \B_{S_N,N, \delta}(\{ \iNxa(\C) \in B(P, \epsilon) \}) = - \ERS[P|\Poisson^m].
\end{equation}
\end{prop}
\begin{proof} First, let us replace $\textbf{B}_{S_N,N,\delta}$ by $\Poisson^m$, the law of a Poisson point process of intensity $m$. It follows from \cite[Theorem 3.1]{Georgii1} that
\[
\lim_{\epsilon \t0} \lim_{N \ti} N^{-2\delta} \log \Poisson^m(\{ \iNxa(\C) \in B(P, \epsilon) \}) = - \ERS[P|\Poisson^m]
\]
The probability under $\Poisson^m$ of having $S_N$ points in $C(0,{N^{\delta}})$ is given by $e^{-mN^{2\delta}} \frac{ (mN^{2\delta})^{S_N}}{S_N!}$ and Stirling's formula yields
\[
N^{-2\delta} \log \left( e^{-mN^{2\delta}} \frac{ (mN^{2\delta})^{S_N} }{S_N!} \right) = - m +  \frac{S_N}{N^{2\delta}} \log \frac{mN^{2\delta}}{S_N} + \frac{S_N}{N^{2\delta}} + o(1).
\]
Since $S_N \sim_{N} mN^{2\delta}$ the right-hand side is $o(1)$. It is not hard to conclude that \eqref{Sanovc} holds, using the fact that $P$ itself is of intensity $m$. For more details we refer to \cite[Section 7.2]{LebSer} where a similar result is proven.
\end{proof}

\section{Preliminary considerations on the energy} \label{sec:preliminaries}
\subsection{Monotonicity estimates}
\paragraph{Almost monotonicity in $\eta$ of the local energy.}
The next lemma expresses the fact that the limit $\eta \t0$ in \eqref{REfinite} is almost monotonous.
\begin{lem} \label{lem:monoton1}
Let $(\C, \mu)$ be a neutral electric system with $N$ points and $\Eloc$ be the associated local electric field. We have, for any $0 < \eta < \eta_1 < 1$,
\begin{equation} \label{monoton1}
\left(\int_{\R^2} |\Eloc_{\eta}|^2 + N \log \eta\right) - \left(\int_{\R^2} |\Eloc_{\eta_1}|^2 + N \log \eta_1 \right) \succeq  -N \|\mu\|_{\infty} \eta_1.
\end{equation}
\end{lem}
\begin{proof}
This is \cite[Lemma 2.3]{PetSer}.
\end{proof}
Let us note that, integrating by parts, we may re-write $\int_{\R^2} |\Eloc_{\eta}|^2$ as $\int_{\R^2} - \Hloc_{\eta} \Delta \Hloc_{\eta}$ and \eqref{monoton1} is really a monotonicity estimate for $\left(\int_{\R^2} - \Hloc_{\eta} \Delta \Hloc_{\eta} + N \log \eta\right)$ as $\eta$ varies.

\paragraph{A localized monotonicity estimate.}
\begin{lem} \label{lem:monoton3}
Let $(\C, \mu)$ be an electric system and $\Eloc$ be the associated local electric field. Let $R_2 > 10$ and let $\Nin$ be the number of points of $\C$ in $C_{R_2}$. We let also $\Nbou$ be the number of points of $\C$ in $C_{R_2} \backslash C_{R_2-5}$. We have for any $0 < \eta_1 < \eta_0 < 1$,
\begin{multline} \label{monotongen}
\left(\int_{C_{R_2}} |\Eloc_{\eta_1}|^2 - \Nin \log \eta_1 \right) - \left(\int_{C_{R_2}} |\Eloc_{\eta_0}|^2 - \Nin \log \eta_0 \right) \succeq - \Nin ||\mu||_{\infty} \eta_0 \\ +  \Nbou \log \eta_1
+ (\log \eta_1 -1) \int_{C_{R_2} \backslash C_{R_2-5}} \left(|\Eloc_{\eta_0}|^2 + |\Eloc_{\eta_1}|^2\right).
\end{multline}
\end{lem}
\begin{proof}
It follows from the proof of \cite[Lemma 2.4]{PetSer}, see e.g. \cite[Equation 2.29]{PetSer} and the one immediatly after.
\end{proof}

\paragraph{Almost monotonicity with no points near the boundary.}
\begin{lem} \label{lem:monoton2}
Let $(\C,\mu)$ be an electric system in $\R^2$ and $\Eloc$ be the associated local electric field. Let $0 < R_2$ and let $0 <\eta_1 < 1$ be such that the smeared out charges at scale $\eta_1$ do not intersect $\partial  C_{R_2} $ i.e. $
\bigcup_{p \in \C} B(p, \eta_1) \cap \partial C_{R_2} = \emptyset.$
Let us denote by $\Nin$ the number of points in $C_{R_2}$. Then we have for any $\eta \leq \eta_1$
\begin{equation} \label{monoton2}
\left(\int_{C_{R_2}} |\Eloc_{\eta}|^2 +  \Nin \log \eta\right)  - \left(\int_{C_{R_2}} |\Eloc_{\eta_1}|^2 +  \Nin \log \eta_1\right) \succeq - \Nin \eta_1 ||\mu||_{\infty}
\end{equation}
\end{lem}
\begin{proof}
We postpone the proof to Section \ref{sec:preuvemonot}.
\end{proof}

\subsection{Discrepancy estimates}
\begin{lem} \label{lem:discr}
Let $N \geq 1$, and let $\C$ be a finite point configuration in $\R^2$. Let $E$ be a gradient electric field in $\Elec(\C, \mupeq)$. For any $R > 0$, let $\D_R$ be the discrepancy $\D_R := \int_{C_R} (d\C - d\mupeq)$ in $C_R$. 

For any $\eta \in (0,1)$ we have
\begin{equation} \label{estimdiscr}
\D^2_R \min \left(1, \frac{\D_R}{R^2} \right) \preceq \int_{C_{2R}} |E_{\eta}|^2.
\end{equation}
\end{lem}
\begin{proof}
This follows from \cite[Lemma 3.8]{RougSer}.
\end{proof}

As a corollary, we see that if a good control holds at scale $\delta$, then the discrepancies are controlled at smaller scales.
\begin{lem} \label{lem:discr2}
Let $0 < \delta \leq \frac{1}{2}$ and let us assume that a good control holds at scale $\delta$. Then for any $R \in  (\frac{1}{2} N^{\deltap}, 2 N^{\deltap})$ with $\frac{\delta}{2} < \deltap < \delta$ we have
\begin{equation} \label{discr-contr}
|\D_R| \ll_{\deltap} N^{\frac{4}{3} \delta}.
\end{equation}
\end{lem}
\begin{proof}
Let us apply Lemma \ref{lem:discr} with $\eta = 1/2$, taking $E$ to be the local electric field $\Eloc$. It yields
\[
\D^2_R \min \left(1, \frac{\D_R}{R^2} \right) \preceq \int_{C_{2R}} |\Eloc_{1/2}|^2.
\]
Using the good control on the energy at scale $\delta$ (in the sense of Definition \ref{def:goodcontrol}) we have $\int_{C_{2R}} |\Eloc_{1/2}|^2 + \Nn_{2R} \log (1/2) \preceq N^{2\delta}$, and $\Nn_{2R}$ is itself $\preceq_{\deltap} N^{2 \delta}$. We thus obtain that $
\D^2_R \min \left(1, \frac{\D_R}{R^2} \right) \preceq_{\deltap} N^{2\delta}$. Then, elementary considerations imply that if $\frac{\delta}{2} < \deltap < \delta$, we have $\D^3_R \preceq_{\deltap} N^{2(\deltap +\delta)}$ which yields \eqref{discr-contr}.
\end{proof}

\paragraph{Application: number of points in a square.}
\begin{lem} \label{lem:discr3}
Let $0 < \delta \leq \frac{1}{2}$ and let us assume that a good control holds at scale $\delta$. Then we have, for any $R \in  (\frac{1}{2} N^{\deltap}, 2 N^{\deltap})$ with $ \frac{2}{3} \delta < \deltap < \delta$, and for any $z_0 \in \mathring{\Sigma}$, letting $\Nnz_R := \int_{C(z'_0, R)} d\C$
\begin{equation} \label{nbpointsNR}
\left| \Nnz_R - \meq(z_0) R^2 \right| \ll_{\deltap} N^{2\deltap}.
\end{equation}
\end{lem}
\begin{proof}
We have by definition $\int_{C(z'_0, R)} d\mupeq = \int_{C(z'_0, R)} \meq(z_0 + tN^{-1/d}) dt$. Using the Hölder assumption \eqref{assum-Holder} we get $\left| \int_{C(z'_0, R)} d\mupeq - R^2 \meq(z_0)\right| \preceq N^{2\deltap + \kappa (\deltap - \frac{1}{2})}$.
Since $\kappa > 0$ and $\deltap < 1/2$ we get
$\left| \int_{C(z'_0, R)} d\mupeq - R^2 \meq(z_0)\right| \ll N^{2\deltap}$. 
Lemma \ref{lem:discr2} yields $\left|\Nnz_R - \int_{C(z'_0, R)} d\mupeq\right| \ll_{\deltap} N^{\frac{4}{3} \delta}$. 
Combining these two inegalities we see that if $\deltap > \frac{2}{3} \delta$ then \eqref{nbpointsNR} holds.
\end{proof}

\subsection{Minimality of local energy against decaying fields}
\begin{lem}  \label{lem:minilocale} Let $K \subset \R^2$ be a compact set with piecewise $C^1$ boundary. Let $(\C,\mu)$ be a neutral electric system in $K$.
Let  $\Eloc$ be the local electric field associated to $(\C, \mu)$ and let $E \in \Eleco(\C, \mu)$. For any $0<\eta <1$ we have
\begin{equation}\label{comparloc}
\int_{\R^{2}}  |\Eloc_{\eta}|^2 \leq \int_{\R^{2}}  |E_{\eta}|^2.
\end{equation}
\end{lem}
The proof is very similar that of \cite[Lemma 3.12]{LebSer} and we postpone it to Section \ref{sec:annexe}.

\subsection{The screening lemma}
\begin{lem} \label{lem:screening}
There exists $C > 0$ universal such that the following holds. 

Let $z \in \R^2$ and for any $N$ let $(\C,\mu)$ be an electric system in $\R^2$, let $0 < \deltatp < \deltapp < \deltap < 1/2$ and let $R_1, R_2$ positive be such that, letting $\eta_1 := N^{-10}$, we have
\begin{enumerate}
\item $R_2 \in [N^{\deltap} + N^{\deltapp}, N^{\deltap} + 2N^{\deltapp}]$ and $R_1 \in [R_2 - 3N^{\deltatp}, R_2 - 2N^{\deltatp}]$.  
\item The smeared out charges at scale $\eta_1$ do not intersect $\partial C(z,R_2)$, i.e. $$\bigcup_{p \in \C}^N D(p, \eta_1) \cap \partial  C(z,R_2) = \emptyset.$$
\item $\Nint$ is an integer, where $\Nint := \int_{C(z,R_1)} d\mu$.
\item Letting $\Nmid := \int_{C(z,R_2) \backslash C(z,R_1)} d\C$, it holds $\Nmid \ll N^{2\deltap}$.
\end{enumerate}

Let us assume that $\mu$ satisfies $0 < \um \leq \mu \leq \om$ on $C_{R_2} \backslash C_{R_1}$ and that furthermore there exists $\Chol > 0$ such that
\begin{equation} \label{assum:muHolder}
\forall (x,y) \in (C_{R_2} \backslash C_{R_1})^2, |\mu(x) - \mu(y)| \leq \Chol |x-y|^{\kappa},
\end{equation}
where $\kappa$ is as in \eqref{assum-Holder}.
Let $E$ be in $\Elec(\C, \mu, \R^2 \backslash C_{R_1})$ and let $M := \int_{\partial C_{R_2}} |E_{\eta}|^2$.

If the following inequality is satisfied
\begin{equation} \label{scrineg}
M \leq C \min(\um^2,1) N^{3\deltatp},
\end{equation}
then there exists a measurable family $\AtranN$ of point configurations such that for any $\Ctran \in \AtranN$ 
\begin{enumerate} 
\item The configuration $\Ctran$ is supported in $C_{R_2} \backslash C_{R_1}$.
\item The configuration $\Ctran$ has $\Next$ points where (with $\vec{n}$ the unit normal vector)
\begin{equation} \label{def:Next}
\Next := \int_{\partial C_{R_2}} E_{\eta_1} \cdot \vec{n} - \int_{C_{R_2} \backslash C_{R_1}} d\mu.
\end{equation}
\item The points of $\Ctran$ are well-separated from each other and from the boundaries
\begin{equation} \label{Cextbiensep}
\min_{p_1 \neq p_2 \in \Ctran} |p_1 - p_2| \succeq \om^{-1/2}, \quad \min_{p \in \Ctran} \dist(p, \partial C_{R_2} \cup \partial C_{R_1}) \succeq \om^{-1/2}.
\end{equation}
\item There exists an electric field $\Etran \in \Elec(\Ctran, \mu, C_{R_2} \backslash C_{R_1})$ such that
\begin{enumerate}
\item We have 
\begin{equation} \label{Eextaubord}
\Etran_{\eta_1} \cdot \vec{n} = \begin{cases} E_{\eta_1} \cdot \vec{n} & \text{ on } \partial C_{R_2} \\
0 & \text{ on } \partial C_{R_1}
\end{cases}.
\end{equation}
\item The energy of $\Etran$ is bounded by
\begin{equation} \label{controleenergieEext}
\int_{C_{R_2} \backslash C_{R_1}} |\Etran_{\eta_1}|^2 \preceq N^{\deltatp} M + N^{\deltap + 3\deltatp} \Chol N^{\kappa \deltatp} + N^{\deltap + \deltatp} \log N.
\end{equation}
\end{enumerate}
\end{enumerate}

Moreover the volume of $\AtranN$ is bounded below as follows
\begin{equation} \label{AextNvol}
\log \Leb^{\otimes \Next}(\AtranN) \succeq - (\om N^{\deltap+\deltatp} + N^{\deltap} + M) \log \om. 
\end{equation}

\end{lem}
\begin{proof}
The result is inspired from the “screening lemmas” of \cite[Prop. 6.4]{SS2d}, \cite[Prop 6.1]{PetSer} and \cite[Prop 5.2]{LebSer}. Our setting is slightly simpler because we have ruled out the possibility of having point charges close to the boundary $\partial C_{R_2}$. Here we screen the electric field “from the inside” (a similar procedure is used in \cite{Rot-NodSer}) whereas the aforementioned screening results were constructing a field $E$ such that $E \equiv 0$ outside a certain hypercube.  Another difference is that in the present lemma we really need to deal with a variable background $\mu$. 

In the rest of the proof we set $l = N^{\deltatp}$.
\\
\textbf{Step 1.} \textit{Subdividing the domain.}
We claim that we may split $C_{R_2} \backslash C_{R_1}$ into a finite family of rectangles $\{H_i\}_{i \in I}$ with sidelengths in $[l/2, 2l]$ such that letting
$\bar{m}_i := \frac{1}{|H_i|} \int_{H_i} d\mu$ and 
\begin{equation}
\label{def:mi} m_i := \bar{m}_i + \frac{1}{|H_i|} \frac{1}{\cds} \int_{\partial C_{R_2} \cap \partial H_i}  E_{\eta} \cdot \vec{n},
\end{equation}
we have for any $i \in I$
\begin{equation} \label{mi:prop1}
\left|m_i - \bar{m}_i \right| < \frac{1}{2} \um, \quad m_i |H_i| \in \N.
\end{equation}

The fact that we may split $C_{R_2} \backslash C_{R_1}$ into a finite family of rectangles $\{H_i\}_{i \in I}$ with sidelengths in $[l/2, 2l]$ is elementary. Using Cauchy-Schwarz's inequality we see that
\[
\left| \frac{1}{|H_i|} \int_{\partial C_{R_2} \cap \partial H_i}  E_{\eta} \cdot \vec{n} \right| \preceq \frac{1}{l^{3/2}} \left(\int_{\partial C_{R_2}} |E_{\eta}|^2\right)^{1/2} = M^{1/2} l^{-3/2}
\]
hence assuming \eqref{scrineg} (with $C$ large enough) we have $\left|m_i - \bar{m}_i \right| < \frac{1}{2} \um$ for any tiling of $C_{R_2} \backslash C_{R_1}$ by rectangles of sidelengths $\in [l/2, 2l]$. It remains to show that we may obtain a tiling such that $m_i |H_i| \in \N$. We have by definition
\[
m_i |H_i| = \int_{H_i} d\mu + \frac{1}{\cds} \int_{\partial C_{R_2} \cap \partial H_i}  E_{\eta} \cdot \vec{n}.
\]
Increasing the sidelengths of $H_i$ by $t$ (with $t \leq l/10$) increases $\int_{H_i} d\mu$ by a quantity $\succeq \um l t$ whereas it changes $\int_{\partial C_{R_2} \cap \partial H_i}  E_{\eta} \cdot \vec{n}$ by a quantity $\preceq \sqrt{tM}$. We thus see that if \eqref{scrineg} holds, up to modifying the boundaries of $H_i$ a little bit (e.g. changing the sidelengths by a quantity less than $l/10$) we can ensure that each $m_i |H_i| \in \N$.

We may then subdivide further each rectangle $H_i$ into a finite family of rectangles $\{R_{\alpha}\}_{\alpha \in I_i}$ which all have an area $m_i^{-1}$ and sidelengths bounded above and below by universal constants times $m_i^{-1/2}$ (for a proof of this fact, see \cite[Lemma 6.3.]{PetSer}). Let us emphasize that since $\mu$ is bounded above we have $m_i^{-1} \preceq 1$.
\\
\textbf{Step 2.} \textit{Defining the transition field and configuration.}
For $i \in I$ we let $E^{(1,i)} := \nabla \h^{(1,i)}$ where $\h^{(1,i)}$ is the solution to 
\[
- \Delta \h^{(1,i)} = \cds (m_i -  \mu) \text{ in } H_i,
\]
such that $\nabla \h^{(1,i)} \cdot \vec{n} = - E_{\eta_1} \cdot \vec{n}$ on each side of $H_i$ which is contained in $\partial C_{R_2}$ and $\nabla \h^{(1,i)} \cdot \vec{n}  = 0$ on the other sides.

We also let, for any $\alpha \in I_i$, $\h^{(2,\alpha)}$ be the solution to
\[
- \Delta \h^{(2,\alpha)} = \cds \left( \delta_{p_{\alpha}} - m_i \right) \text{ in } R_{\alpha}, \nabla \h^{(2,\alpha)} \cdot \vec{n}  = 0 \text{ on } \partial R_{\alpha},
\]
where $p_{\alpha}$ is the center of $R_{\alpha}$. We define $E^{(2,i)}$ as $E^{(2,i)} : = \sum_{\alpha \in I_i} \nabla \h^{(2,\alpha)}$. Finally we define the \textit{transition field} $\Etran$ as $\Etran := \sum_{i \in I} E^{(1,i)} + E^{(2,i)}$, and the \textit{transition configuration} $\Ctran$ as $\Ctran := \sum_{i \in I} \sum_{\alpha \in I_i} \delta_{p_{\alpha}}$.
It is easy to see that
\begin{equation} \label{EtranCtran}
- \div(\Etran) = \cds \left( \Ctran - \mu \right) \text{ in } C_{R_2} \backslash C_{R_1}, \quad \Etran_{\eta_1} \cdot \vec{n} = \begin{cases} E_{\eta_1} \cdot \vec{n} & \text{ on } \partial C_{R_2} \\
0 & \text{ on } \partial C_{R_1}
\end{cases}.
\end{equation}
In particular \eqref{def:Next} and \eqref{Eextaubord} hold.
\\
\textbf{Step 3.} \textit{Controlling the energy.}
For any $i \in I$ the energy of $E^{(1,i)}$ can be bounded using Lemma \ref{lem:screenestim} as follows
\[
\int_{H_i} |E^{(1,i)}|^2 \preceq l \int_{\partial H_i \cap \partial C_{R_2}} |E_{\eta}|^2 + l^{4} ||\mu - \bar{m}_i||_{L^{\infty}(H_i)}^2,
\]
and using the Hölder assumption \eqref{assum:muHolder} on $\mu$ we have $||\mu - \bar{m}_i||_{L^{\infty}(H_i)}^2 \preceq \Chol l^{\kappa}$ hence
\[
\int_{H_i} |E^{(1,i)}|^2 \preceq l \int_{\partial H_i \cap \partial C_{R_2}} |E_{\eta}|^2 + \Chol l^{4+\kappa}.
\]
For any $\alpha \in I_i$ we also have, again by standard estimates
\[
\int_{R_{\alpha}} |\nabla h^{(2,\alpha)}_{\eta_1}|^2 \preceq -\log \eta_1 \preceq \log N \text{ by choice of $\eta_1$.}
\]
The number of rectangles $R_{\alpha}$ for $\alpha \in I_i$, $i \in I$ is bounded by the volume of $C_{R_2} \backslash C_{R_1}$ hence is $\preceq N^{\deltap + \deltatp} $. We deduce that
\[
\int_{C_{R_2} \backslash C_{R_1}} |\Etran_{\eta_1}|^2 \preceq l \int_{\partial C_{R_2}}  |E_{\eta}|^2 + \#I \Chol l^{4+\kappa} + N^{\deltap + \deltatp} \log N,
\]
where $\#I$ denotes the cardinality of $I$. We may observe that $\#I \preceq N^{\deltap + \deltatp} l^{-2}$ and get (using that $l = N^{\deltatp}$)
\[
\int_{C_{R_2} \backslash C_{R_1}} |\Etran_{\eta_1}|^2 \preceq N^{\deltatp} M + N^{\deltap + 3\deltatp} \Chol N^{\kappa \deltatp} + N^{\deltap + \deltatp} \log N.
\]
which yields \eqref{controleenergieEext}.
\\
\textbf{Step 4.} \textit{Constructing a family.}
As was observed in \cite[Remark 6.7]{PetSer}, \cite[Proposition 5.2]{LebSer}, we may actually construct a whole family of configurations $\Ctran$ and associated electric fields $\Etran$ such that \eqref{EtranCtran} and \eqref{controleenergieEext} hold. Indeed, since $\mu \leq \om$ the sidelengths of $R_{\alpha}$ are $\succeq \om^{-1/2}$ hence we may move each of the points $p_{\alpha}$ (for $\alpha \in I_i$, $i \in I$) arbitrarily within a disk of radius $\frac{1}{10} \om^{-1/2}$ and proceed as above (so that \eqref{Cextbiensep} is conserved). This creates a volume of configurations of order 
$\om^{-\Ntran}$. To get \eqref{AextNvol} it suffices to observe that
\[
\Ntran \preceq \om N^{\deltap+\deltatp} + N^{\deltap} + M,
\]
which follows from \eqref{def:Next} by applying the Cauchy-Schwarz inequality and using the fact that $R_2^2-R_1^2 \preceq N^{\deltap+\deltatp}$.
\end{proof}

\section{A large deviation upper bound} \label{sec:LDPUB}
Given a point $z_0$ and a scale $\deltap$ we localize the energy $\wN(\XN)$ by splitting it between the “interior part” (which, roughly speaking, corresponds to the $L^2$-norm of $\Eloc$ on $C(z'_0, N^{\deltap})$) and the “exterior part”, then we replace both quantities by smaller ones which are independent (this corresponds to separating variables). Roughly speaking, it allows us to localize the problem on $C(z'_0, N^{\deltap})$ by deriving a “local energy” (which corresponds to the interior part) and a “local partition function” (the exponential sum of contributions of the exterior part). This lower bound on the energy will be complemented by a matching upper bound in Section \ref{sec:LDPLB}.
  
\subsection{Definition of good interior and exterior boundaries and energies}
The decomposition between interior and exterior part will be done at the boundary of some “good” square, not much larger than $C(z'_0, N^{\deltap})$. We give the definition of good exterior and interior boundaries, with an abuse of notation which is discussed in the next subsection. 
 
\begin{defi}[Exterior boundary] \label{def:extboundary}
Let $1/2 > \delta > \deltap > \deltapp > 0$ and $\eta_0 > 0$ be fixed, let $N \geq 1$ and let $\eta_1 := N^{-10}$. Let $z_0 \in \mathring{\Sigma}$, let $R_2 > 0$ and let $\Xou$ be a point configuration in $C(z'_0,R_2)^c$. 
\\
\textbf{Exterior fields.} We say that $E$ is in $\Elecou(\Xou)$ if $E$ is in $\Eleco(\C, \mupeq, \R^2)$ for some point configuration $\C$ with $N$ points such that $\C = \Xou$ on $C(z,R_2)^c$.
\\
\textbf{Good exterior boundary.}
Let $E \in \Elecou(\Xou)$. We say that $\partial C(z'_0,R_2)$ is a \textit{good exterior boundary} for $E$ if the following conditions are satisfied:
\begin{enumerate}
\item  We have
\begin{equation} \label{R2bon}
R_2 \in [N^{\deltap} + N^{\deltapp}, N^{\deltap} + 2N^{\deltapp}].
\end{equation}
\item The smeared out charges at scale $\eta_1$ do not intersect $\partial C(z'_0,R_2)$
\begin{equation} \label{nointersection}
\bigcup_{p \in \C} D(p, \eta_1) \cap \partial  C(z'_0,R_2) = \emptyset.
\end{equation}
\item The energy near $\partial C(z'_0, R_2)$ is controlled as follows 
\begin{align}
\label{bonbord3} \int_{C(z'_0,R_2) \backslash C(z'_0,R_2-5)} |E_{\eta_0}|^2 & \preceq N^{2\delta - \deltapp} |\log \eta_0|,\\
\label{bonbord2} \int_{C(z'_0,R_2) \backslash C(z'_0,R_2-5)} |E_{\eta_1}|^2 & \preceq N^{2\delta - \deltapp} \log N,\\
\label{bonbord1}
\int_{\partial C(z'_0, R_2)} |E_{\eta_1}|^2 & \preceq N^{2\delta - \deltapp} (\log N)^2.
\end{align}
\item We have, letting $\Nbou := \int_{C(z'_0,R_2) \backslash C(z'_0,R_2-5)} d\C$  \begin{equation}
\label{nboucont} \Nbou \ll N^{2\deltap}.
\end{equation}
\item We have, letting $\Nin := \int_{C(z'_0,R_2)} d\C$
\begin{equation}
\label{nincont} \Nin \preceq N^{2\deltap}.
\end{equation}
\item We have, letting $\Nxa := \int_{C(z'_0, N^{\deltap})} d\C$
\begin{equation} \label{ninmnxa}
\Nin - \Nxa \ll N^{2 \deltap} \text{ and } |\Nxa - \meq(z_0) N^{2\deltap}| \ll N^{2\deltap}.
\end{equation}
This very last inequality does not depend on $R_2$, but it is convenient to include it in the definition of the exterior boundary.
\end{enumerate}

\textbf{Best exterior energy.} 
Let $\Nou$ denote the number of points of $\Xou$ in $C(z'_0,R_2)^c$. We define $\Fou(\Xou)$ as
\begin{equation} \label{def:Fou}
\Fou(\Xou) := \frac{1}{2\pi} \min_{E} \lim_{\eta \t0} \left( \int_{C(z'_0,R_2)^c} |E_{\eta}|^2 + \Nou \log \eta \right)
\end{equation}
where the $\min$ is taken over the set of electric fields $E$ satisfying $E \in \Elecou(\Xou)$ and such that $\partial C(z'_0,R_2)$ is a good exterior boundary for $E$ (if there is no such field we set $\Fou(\Xou) = +\infty$). The minimum is achieved because on the one hand $E \mapsto |E_{\eta}|^2$ is coercive for the weak $\Lploc$ topology, and on the other hand the $L^2$ norm is coercive and lower semi-continuous for the weak $L^2$ topology. 
\end{defi}

\begin{defi}[Interior boundary] \label{def:intboundary} 
Let $1/2 > \deltap > \deltapp > \deltatp > 0$ and $\eta_0  > 0$ be fixed. Let $N \geq 1$ and let $z_0 \in \mathring{\Sigma}$, let $R_1 > 0$, let $R_2 > 0$ such that \eqref{R2bon} holds and let $\Xin$ be a point configuration in $C(z'_0,R_2)$. 

\textbf{Interior fields.} Let $E \in \Elec(\R^2)$. We say that $E$ is in $\Elecin(\Xin)$ if $E$ is in $\Eleco(\C, \mupeq, \R^2)$ for some $\C \in \config(\R^2)$ such that $\C = \Xin$ on $C(z'_0,R_2)$.

\textbf{Good interior boundary.}
We say that $\partial C(z'_0,R_1)$ is a \textit{good interior boundary} for $\Xin$ if
\begin{enumerate}
\item We have
\begin{equation} \label{R1bon}
R_1 \in [R_2 - 3N^{\deltatp}, R_2 - 2N^{\deltatp}].
\end{equation}
\item $\Nint$ is an integer, where $\Nint := \int_{C(z'_0,R_1)} d\mupeq$.
\item Letting $\Nmid := \int_{C(z'_0,R_2) \backslash C(z'_0,R_1)} d\C$, it holds
\begin{equation} \label{Nextcontrole}
\Nmid \ll N^{2\deltap}.
\end{equation}
\end{enumerate}

\textbf{Best interior energy.} 
Let $\Nin$ denote the number of points of $\Xin$ in $C(z'_0,R_2)$. For any $0 < \eta_0 < 1$ we define $\Fin_{\eta_0}$ as
\begin{equation}
\Fin_{\eta_0}(\Xin) := \frac{1}{2\pi} \min_{E} \left(\int_{C(z'_0,R_2)} |\Eloc_{\eta_0}|^2 + \Nin \log \eta_0\right),
\end{equation}
where the minimum is taken over the set of electric fields $E$ such that $E \in \Elecin(\Xin)$ (if $\Elecin(\Xin)$ is empty we set $\Fin(\Xin) = + \infty$).
\end{defi}

\subsection{Finding good boundaries}
The conditions \eqref{bonbord3}, \eqref{bonbord2}, \eqref{bonbord1}, \eqref{nboucont}, \eqref{nincont}, \eqref{ninmnxa}, \eqref{Nextcontrole} are \textit{asymptotic} as $N \ti$, in particular they do not make sense for a finite $N$ (nonetheless, \eqref{R2bon}, \eqref{nointersection} and \eqref{R1bon} do). Strictly speaking one thus has to consider sequences $R_2 = R_2(N)$ and $R_1 = R_1(N)$.

\begin{lem} \label{lem:intermedUB}
Let $1/2 \geq \delta > \deltap > \deltapp > \deltatp > 0$ and $\eta_0  > 0$ be fixed, with $\deltap > 2\delta/3$.
 Let $z_0 \in \mathring{\Sigma}$ and $\eta_0 > 0$ be fixed. Assume that good control at scale $\delta$ holds and let $\deltap \in (\delta/2, \delta)$. With $\deltap$-overhelming probability, there exists $R_1, R_2$ such that, letting
$\Xin = \XpN \cap C(z'_0,R_2)$ and $\Xou = \XpN \cap C(z'_0,R_2)^c$, we have
\begin{enumerate}
\item $\partial C(z'_0,R_2)$ is a good exterior boundary for $\Xou, \Eloc$.
\item $\partial C(z'_0,R_1)$ is a good interior boundary for $\Xin$.
\end{enumerate}

\end{lem}

\begin{proof}[Proof of Lemma \ref{lem:intermedUB}]
First we look for a good exterior boundary $\partial C(z'_0, R_2)$. The good control at scale $\delta$ implies that
\begin{align} \label{goodcontrolapp1}
\int_{C(z'_0, N^{\delta})} |\Eloc_{\eta_1}|^2 \preceq_{\deltap} N^{2\delta} (1 + |\log \eta_1|) \preceq N^{2\delta} \log N,\\
\label{goodcontrolapp2}
\int_{C(z'_0, N^{\delta})} |\Eloc_{\eta_0}|^2 \preceq_{\deltap} N^{2\delta} (1 + |\log \eta_0|) \preceq N^{2\delta} |\log \eta_0|.
\end{align}
In view of \eqref{goodcontrolapp1}, \eqref{goodcontrolapp2}, by the pigeonhole principle, we may find (with $\deltap$-overhelming probability) an interval $[R-10, R+10]$ included in $[N^{\deltap} + N^{\deltapp}, N^{\deltap} + 2N^{\deltapp}]$ such that
\begin{equation} \label{controleenergbound}
\int_{C(z'_0,R+10) \backslash C(z'_0,R-10)} \left(|\Eloc_{\eta_1}|^2 + |\Eloc_{\eta_0}|^2\right) \preceq N^{2\delta - \deltapp} (\log N + |\log \eta_0|).
\end{equation}
We may find $N^{2}$ disjoint strips of width $2N^{-2} (\log N)^{-1}$ in $[R-8, R+8]$. In view of \eqref{controleenergbound}, there are at least $N^2/2$ such strips on which
the integral of $|\Eloc_{\eta_1}|^2 + |\Eloc_{\eta_0}|^2$ is $\preceq N^{2\delta - \deltapp -2} (\log N + |\log \eta_0|)$. On the other hand there are at most $N$ point charges, thus since $\eta_1 = N^{-10}$ by the pigeonhole principle we may moreover assume that no smeared out charge (at scale $\eta_1$) intersects the strips. Finally a mean value argument on one of these strips shows that we may find $R_2$ such that \eqref{bonbord1} and \eqref{nointersection} holds. By \eqref{controleenergbound} we also have \eqref{bonbord3} and \eqref{bonbord2}. 

Next, we look for a good interior boundary $\partial C(z'_0,R_1)$. Since $z_0$ is in the interior of $\Sigma$, the density $\mpeq$ is bounded above and below by positive constants on $C(z'_0, N^{\delta})$ (for $N$ large enough) and thus the derivative of $R \mapsto \int_{C(z'_0,R)} d\mpeq$ is bounded above and below by (positive) constants times $N^{\deltap}$ on $C(z'_0,2N^{\deltap})$. Hence we may find $R_1 \in [R_2 - 3N^{\deltatp}, R_2 - 2N^{\deltatp}]$ such that $\int_{C(z'_0,R_1)} d\mupeq$ is an integer, hence the first two points of the definition of a good interior boundary are satisfied.

We have \eqref{nincont} with $\deltap$-overhelming probability according to the good control at scale $\delta$. Since $\deltap > 2\delta/3$, the discrepancy estimates of Lemma \ref{lem:discr3} imply that, up to an error $\ll_{\deltap} N^{2\deltap}$ we have 
\begin{align*}
\Nbou &= \meq(z_0) (R_2^2 - (R_2-10)^2) \preceq_{\deltap} N^{\deltap} \ll N^{2\deltap} \\
\Next &= \meq(z_0) (R^2_2 - R_1^2) \preceq_{\deltap} N^{\deltap+\deltapp} \ll N^{2\deltap},
\end{align*}
which proves \eqref{nboucont} and \eqref{Nextcontrole}. We obtain \eqref{ninmnxa} with similar arguments.
\end{proof}

\subsection{A first LDP upper bound}

\begin{prop} \label{prop:LDPUB}
Let $1/2 \geq \delta > \deltap > \deltapp > \deltatp > 0$ be fixed with $\deltap > 2\delta/3$ and $2\delta - \deltapp < \deltap$. Let $z_0 \in \mathring{\Sigma}$ and $0 < \eta_0 < 1$ be fixed. Assume that a good control holds at scale $\delta$. 

For any $P \in \probas_{s}(\config)$ and any $\epsilon > 0$ we have
\begin{multline} \label{LDPUB}
\log \PNbetaxb \left(B(P,\epsilon)\right)\leq - \log \KNbeta \\ + \log \max_{R_1,R_2, \Nou} {N \choose \Nou} \left(\int_{\iNxb(\Xin) \in B(P, \epsilon)} \hspace{-1cm} e^{-\hal \beta \Fin_{\eta_0}(\Xin)} d\Xin  \right) \left(\int e^{-\hal \beta (\Fou(\Xou) + N \Zeta(\Xou))} d\Xou\right) \\
+ N^{2\deltap} O(\eta_0).
\end{multline}
\end{prop}
Let us first give some precisions about \eqref{LDPUB}. The $\max$ on $R_1, R_2, \Nou$ is restricted to the set of $\{R_1, R_2\}$ such that \eqref{R2bon} and \eqref{R1bon} hold, with $\Nou$ between $1$ and $N$. Once $\Nou$ is fixed, we let $d\Xin = dx_1 \dots dx_{\Nin}$ and $d\Xou = dx_1 \dots dx_{\Nou}$, with $\Nou + \Nin = N$.

\begin{proof}
Using the definition of $\PNbetaxb$ and of the Gibbs measure $\PNbeta$ we have
\begin{equation} \label{GibbsUB}
\PNbetaxb \left(B(P,\epsilon)\right) = \frac{1}{\KNbeta} \int_{(\iNxb)^{-1} (B(P, \epsilon))} e^{-\hal \beta (\wN(\nu'_N) + N\Zeta(\nu_N))} d\XN. 
\end{equation}

\textbf{Step 1.} \textit{Finding good boundaries.}
We apply Lemma \ref{lem:intermedUB}. With $\deltap$-overhelming probability we obtain $R_1, R_2$ such that \eqref{R2bon} and \eqref{R1bon} holds, and such that, letting
$\Xin = \XpN \cap C(z'_0,R_2)$ and $\Xou = \XpN \cap C(z'_0,R_2)^c$, we have
\begin{enumerate}
\item $\partial C(z'_0,R_2)$ is a good exterior boundary for $\Xou, \Eloc$.
\item $\partial C(z'_0,R_1)$ is a good interior boundary for $\Xin$.
\end{enumerate}
In order to prove \eqref{LDPUB} we may restrict ourselves, in the right-hand side of \eqref{GibbsUB}, to any event of $\deltap$-overhelming probability, and we will henceforth assume that good boundaries exist.

\textbf{Step 2.} \textit{Splitting the energy $\wN$.} For any $\XN$, let $R_1, R_2$ be as above.
We have, using \eqref{REfinite}
\begin{equation} \label{splittingwNUB}
\wN(\nu'_N) = \lim_{\eta \t0} \left(\int_{C(z'_0,R_2)} |\Eloc_{\eta}|^2 + \Nin \log \eta\right) \\ + \lim_{\eta \t0} \left(\int_{C(z'_0,R_2)^c} |\Eloc_{\eta}|^2 +\Nou \log \eta \right).
\end{equation}
Since \eqref{nointersection} holds we may apply Lemma \ref{lem:monoton2} to $(\XpN, \mupeq)$ and $R_2, \eta_1$ as above. It yields, since $\eta_1 = N^{-10}$,
\begin{equation*}
\lim_{\eta \t0} \left(\int_{C(z'_0,R_2)} |\Eloc_{\eta}|^2 + \Nin \log \eta\right) \geq \left( \int_{C(z'_0,R_2)} |\Eloc_{\eta_1}|^2 + \Nin \log \eta_1\right) + o(1) \text{, as } N \ti.
\end{equation*}
We may then apply Lemma \ref{lem:monoton3} to $(\nu'_N, \mupeq)$ and $R_2$, with $\eta_0, \eta_1$ as above. The number $\Nin$ of points in $C(z'_0,R_2)$ is controlled by \eqref{nincont}, the number $\Nbou$ of points near the boundary is controlled by \eqref{nboucont} and the energy near the boundary is controlled by \eqref{bonbord3} and \eqref{bonbord2}. We obtain
\begin{multline*}
\left(\int_{C(z'_0,R_2)} |\Eloc_{\eta_1}|^2 + \Nin \log \eta_1\right) - \left(\int_{C(z'_0,R_2)} |\Eloc_{\eta_0}|^2 + \Nin \log \eta_0\right) \\ \succeq  - N^{2\deltap} \eta_0 - \Nbou \log N - N^{2\delta - \deltapp} (\log N)^2,
\end{multline*}
which may be simplied (assuming that $2\delta - \deltapp < \deltap$, which will be later ensured by the choice \eqref{choixdelta}) as
\begin{equation} \label{almomontoeta01}
\left(\int_{C(z'_0,R_2)} |\Eloc_{\eta_1}|^2 + \Nin \log \eta_1\right) - \left(\int_{C(z'_0,R_2)} |\Eloc_{\eta_0}|^2 + \Nin \log \eta_0\right) \succeq  - N^{2\deltap} \eta_0.
\end{equation}

Using Definition \ref{def:intboundary} we thus get
\begin{equation} \label{Finbb}
\frac{1}{2\pi} \lim_{\eta \t0} \left(\int_{C(z'_0,R_2)} |\Eloc_{\eta}|^2 + \Nin \log \eta\right) - \Fin_{\eta_0}(\Xin) \succeq - \eta_0 N^{2\deltap}.
\end{equation}

On the other hand, we may write, using Definition \ref{def:extboundary}
\begin{equation} \label{Foubb}
\frac{1}{2\pi} \lim_{\eta \t0} \left(\int_{C(z'_0,R_2)^c} |\Eloc_{\eta}|^2 + \Nou \log \eta \right) \geq \Fou(\Xou).
\end{equation}

Combining \eqref{splittingwNUB}, \eqref{Finbb}, \eqref{Foubb} and inserting them into \eqref{GibbsUB} yields \eqref{LDPUB}.
\end{proof}

\subsection{Good control upper bound}
\begin{lem} \label{lem:GCUB}
Let $1/2 \geq \delta > \deltap > \deltapp > \deltatp > 0$ be fixed with $\deltap > 2\delta/3$. Let $z_0 \in \mathring{\Sigma}$ and $0 < \eta_0 < 1$ be fixed. Assume that a good control holds at scale $\delta$. 

Let us denote by $\mc{E}_{M}$ the event 
\begin{equation}
\label{def:mcEM} \mc{E}_{M} := \left\lbrace \int_{C(z'_0, N^{\deltap})} |\Eloc_{\eta_0}|^2 + \Nxa \log \eta_0 \geq 2\pi  N^{2\deltap} M\right\rbrace.
\end{equation}
We have
\begin{multline} \label{GCUB}
\log \PNbeta \left( \mc{E}_{M} \right)\leq - \log \KNbeta - \frac{\beta}{2} M \\ + \log \max_{R_1,R_2,\Nou} {N \choose \Nou} \left(\int e^{-\beta (\Fou(\Xou) + N \Zeta(\Xou))} d\Xou\right)
+ N^{2\deltap} O(\eta_0).
\end{multline}
\end{lem}
\begin{proof}
We follow the same steps as in the proof of Proposition \ref{prop:LDPUB}, replacing the event $B(P,\epsilon)$ by $\mc{E}_{M}$. 

Let us write
\[
\left(\int_{C(z'_0,R_2)} |\Eloc_{\eta_0}|^2 + \Nin \log \eta_0\right)  \geq \left(\int_{C(z'_0,R_1)} |\Eloc_{\eta_0}|^2 + \Nxa \ \log \eta_0\right) + (\Nin - \Nxa) \log \eta_{\eta_0}
\]

Using \eqref{almomontoeta01} and the definition of $\mc{E}_M$ we get, conditionally to $\mc{E}_{M}$ 
\[
\left(\int_{C(z'_0,R_2)} |\Eloc_{\eta_1}|^2 + \Nin \log \eta_1\right) \geq N^{2\deltap} M + (\Nin - \Nxa) \log \eta_0  + O(N^{2\deltap}).
\]
Since we have $|\Nin - \Nxa|_{\deltap} \ll N^{2\deltap}$ we deduce that (conditionally to $\mc{E}_{M}$), using Lemma \ref{lem:monoton2}
\begin{equation} \label{varGC}
\lim_{\eta \t0} \left(\int_{C(z'_0,R_2)} |\Eloc_{\eta}|^2 + \Nin \log \eta\right) \geq N^{2\deltap} M + O(N^{2\deltap}).
\end{equation}
Using \eqref{varGC} instead of \eqref{Finbb} and arguing as in the proof of Proposition \ref{prop:LDPUB} gives \eqref{GCUB}.
\end{proof}

\subsection{Large deviation upper bound for the interior part}
In \eqref{LDPUB} we have separated the Gibbs measure into its interior and exterior parts. In the next lemma, we give a large deviation upper bound for the interior part, namely
\[
\int_{\iNxb(\Xin) \in B(P, \epsilon)}  e^{-\hal \beta \Fin_{\eta_0}(\Xin)} d\Xin.
\]
Up to technical details, this is a classical application of Varadhan's lemma: the large deviations for the reference point process (without the exponential term) are known from Proposition \ref{prop:Sanov}, and on the other hand the lower semi-continuity of the energy near a random stationary point process $P$ implies that  $\Fin_{\eta_0}(\Xin) \geq \bW_{\meq(z_0)}(P)$ on $\{\iNxb(\Xin) \in B(P, \epsilon)\}$, hence we obtain
\begin{lem} \label{lem:lowerboundint}
Let $1/2 \geq \delta > \deltap > \deltapp > \deltatp > 0$ be fixed with $\deltap > 2\delta/3$, and let us assume that good control at scale $\delta$ holds. We have, for any $R_1, R_2$ satisfying \eqref{R2bon}, \eqref{R1bon}, and any $\Nin$
\begin{equation} \label{lowerboundint}
\limsup_{\eta_0 \t0, \epsilon \t0, N \ti} N^{-2\deltap} \log \int_{\iNxb(\Xin) \in B(P, \epsilon)}  e^{-\hal \beta \Fin_{\eta_0}(\Xin)} d\Xin \leq - \fbeta^{\meq(z_0)}(P).
\end{equation}
\end{lem}
\begin{proof} 

\textbf{Step 1.} \textit{Lower semi-continuity of the energy.}
We claim that
\begin{multline} \label{lsciW}
\liminf_{\eta_0 \t0, \epsilon \t0, N \ti} N^{-2\deltap} \inf\left\lbrace\Fin_{\eta_0}(\Xin), \Xin \in \config (C(z'_0,R_2)), \iNxb(\Xin) \in B(P, \epsilon)\right\rbrace \\ \geq \bW_{\meq(z_0)}(P).
\end{multline}

To prove \eqref{lsciW}, let $E = E(N)$ be a minimizing sequence in \eqref{lsciW}, let $\C = \C(N)$ be the associated point configuration in $C(z'_0, N^{\deltap})$ and let us define
\[\Pelec_N := N^{-2\deltap} \int_{C(z'_0, N^{\deltap})} \delta_{\theta_z' \cdot E} dz', \quad
P_N := N^{-2\deltap} \int_{C(z'_0, N^{\deltap})} \delta_{\theta_z' \cdot \C} dz'.
\]
We have, for any $m \geq 1$, 
\[
\Esp_{\Pelec_N} \left[ \frac{1}{|C_m|} \int_{C_m} |E_{\eta_0}|^2 \right] \leq N^{-2\deltap} \int_{C(z'_0,R_2)} |E_{\eta_0}|^2,
\]
which proves that the push-forward of $\Pelec_N$ by $E \mapsto E_{\eta_0}$ is tight in $L^2(C_m, \R^2)$ for the weak topology, whereas the sequence $\{\Pelec_N\}$ itself is tight in $L^p(C_m, \R^2)$ (indeed when $\eta$ is fixed, the $L^2$-norm of $E_{\eta}$ controls the $L^p$-norm of $E$, this follows easily from Hölder's inequality, see \cite[Lemma 3.9]{RougSer}). On the other hand, the sequence of random point processes $\{P_N\}$ is also tight because the expectation of the number of points in any square is bounded and, up to subsequence extraction, it converges to some $Q \in B(P, \epsilon)$. Denoting by $Q^{\rm{elec}}$ a limit point of $\{\Pelec_N\}$ it is not hard to see that $Q^{\rm{elec}}$ is compatible with $Q$, and by lower semi-continuity of the $L^2$-norm with respect to weak convergence we have
\[
\frac{1}{|C_m|} \int_{C_m} |E_{\eta_0}|^2 dQ^{\rm{elec}} \leq \liminf_{N \ti} N^{-2\deltap}  \int_{C(z'_0,R_2} |E_{\eta_0}|^2.
\]
Up to applying a standard diagonal extraction procedure we may assume that it holds for any $m \geq 1$, hence
\[
\Esp_{Q^{\rm{elec}}} [|E_{\eta_0}|^2] \leq \liminf_{N \ti} N^{-2\deltap}  \int_{C(z'_0,R_2)} |E_{\eta_0}|^2.
\]
Letting $\epsilon \t0$ and arguing as above concerning the tightness of $Q^{\rm{elec}}$ and $Q$ (in $\Lploc$ and in $\config)$ and for the lower semi-continuity of the norm with respect to weak convergence we obtain
\[
\Esp_{\Pelec} [|E_{\eta_0}|^2] \leq \lim_{\epsilon \t0} \liminf_{N \ti} N^{-2\deltap} \int_{C(z'_0,N^{\deltap})} |E_{\eta_0}|^2
\]
where $\Pelec$ is some random electric field compatible with $P$. Letting $\eta_0 \t0$, using \eqref{ninmnxa} and the definition of $\bW_{\meq(z_0)}(P)$ we thus obtain
\[
\bW_{\meq(z_0)}(P) \leq \lim_{\eta_0 \t0, \epsilon \t0} \liminf_{N \ti} N^{-2\deltap} \left(\int_{C(z'_0,N^{\deltap})} |E_{\eta_0}|^2 + \Nxa \log \eta_0 \right).
\]
\\
\textbf{Step 2.} \textit{Large deviations without interactions}.
We claim that, on the other hand
\begin{equation} \label{SanovUBapply}
\limsup_{\epsilon \t0, N \ti} N^{-2\deltap} \log \int_{\iNxb(\Xin) \in B(P, \epsilon)} \1_{C(z'_0,R_2)}(\Xin) d\Xin  \leq - \ERS[P|\Poisson^{\meq(z_0)}].
\end{equation}
The configuration $\Xin$ has $\Nin$ points. Among them, $\Nin - \Nxa$ are not affected by the constraint $\iNxa(\Xin) \in B(P, \epsilon)$ because they belong to $C(z'_0, N^{\deltap})^c$, and are free to move in $C(z'_0, R_2) \backslash C(z'_0, N^{\deltap})$, but we know from \eqref{ninmnxa} that $\Nin - \Nxa \ll N^{2\deltap}$, thus the volume contribution of these points is negligible because $N^{-2\deltap} \log |C(z'_0, R_2) \backslash C(z'_0, N^{\deltap})|^{\Nin - \Nxa} = o(1)$.
On the other hand we know, using \eqref{ninmnxa}, that $\Nxa \sim_{N \ti} \meq(z_0)N^{2\deltap}$ and then \eqref{SanovUBapply} follows from Proposition \ref{prop:Sanov}.
\\
\textbf{Step 3.} \textit{Conclusion.}
Combining \eqref{lsciW} and \eqref{SanovUBapply} yields \eqref{lowerboundint}.
\end{proof}

\subsection{A second LDP upper bound}
Combining Proposition \ref{prop:LDPUB} and Lemma \ref{lem:lowerboundint} and letting $\eta_0 \t0$ we obtain
\begin{prop} \label{prop:LDPUB2}
\begin{equation} \label{LDPUB3}
\lim_{\epsilon \t0} \limsup_{N \ti} N^{-2\deltap} \log \PNbetaxb \left(B(P,\epsilon)\right)
\leq - \fbeta^{\meq(z_0)}(P) + \limsup_{N \ti} N^{-2\deltap} \log \KNbetax
\end{equation}
where we let $\KNbetax$ be such that
\begin{equation} \label{def:KNbetax}
\log \KNbetax :=  - \log \KNbeta + \log \int_{R_1,R_2} \sum_{\Nou} {N \choose \Nou} \int e^{-\hal \beta (\Fou(\Xou) + N \Zeta(\Xou))} d\Xou.
\end{equation}
\end{prop}

\section{Large deviation lower bound} \label{sec:LDPLB}
In this section, we derive a converse estimate to \eqref{LDPUB3}, by showing that splitting the energy as in Proposition \ref{prop:LDPUB} is essentially sharp as far as probabilities of order $\exp(-N^{2\deltap})$ are concerned. 
\subsection{Generating microstates}
In the next lemma, we recall a tool which was introduced in \cite{LebSer}. Given a stationary point process $P$ and a large square $C_R$, Lemma \ref{lem:generermicro} can be thought of as a way of generating a family of point configurations in $C_R$ whose empirical field is close to $P$, whose interaction energy is close to the renormalized energy of $P$, and such that the volume of the family is optimal in view of the specific relative entropy of $P$. 
\begin{lem} \label{lem:generermicro}
Let $z_0 \in \mathring{\Sigma}$ and $0 < \deltapp < \deltap < 1/2$ be fixed. Let $P \in \probas_{s,\meq(z_0)}(\config)$ such that $\bW_{\meq(z_0)}(P)$ and $\ERS[P|\Poisson^{\meq(z_0)}]$ are finite. 

For any $N \geq 1$, let $R_1 > 0$ be such that $R_1 \in (N^{\deltap}, N^{\deltap} + 2N^{\deltapp})$ and $\Nint := \int_{C(z'_0, R_1)} d\mpeq$ is an integer. Moreover let us assume that $\Nint \sim_{N \ti} \meq(z_0) N^{2\deltap}$.

Then there exists a family $\AintN$ of point configurations in $C(z'_0, R_1)$ such that the following properties hold for any $\Cint~\in~\AintN$:
\begin{enumerate}
\item The configuration $\Cint$ has $\Nint$ points in $C(z'_0, R_1)$.
\item The continuous average of $\Cint$ on $C(z'_0, N^{\deltap})$ is close to $P$, i.e.
\begin{equation} \label{presdeP}
\iNxb(\Cint) \in B(P, o(1)) \text{ as } N \ti.
\end{equation}
\item There exists an electric field $\Eint \in \Elec\left(\Cint, \mupeq, C(z'_0, R_1) \right)$ such that 
\begin{enumerate}
\item $\Eint \cdot \vec{n} = 0$ on $\partial C(z'_0, R_1)$, where $\vec{n}$ is the unit normal vector.
\item The energy of $\Eint$ is controlled by $\bW_{\meq(z_0)}(P)$
\begin{equation} \label{bonnenergie1}
\frac{1}{2\pi} N^{-2\deltap}  \lim_{\eta \t0} \left( \int_{C(z'_0, R_1)} |\Eint_{\eta}|^2 + \Nint \log \eta \right) \leq \bW_{\meq(z_0)}(P) + o(1) \text{ as } N \ti,
\end{equation}
uniformly on $\AintN$.
\end{enumerate}
\item The (logarithmic) volume of the family is close to the relative specific entropy of $P$
\begin{equation} \label{bonvolume1}
N^{-2\deltap}\log \Leb^{\otimes \Nint}(\AintN) \geq - \ERS[P|\Poisson^{\meq(z_0)}] + o(1) \text{, as } N \ti.
\end{equation}
\end{enumerate}
\end{lem}
\begin{proof}
This follows from the analysis of \cite[Section 6]{LebSer}. Let us sketch the main steps here. 

We fix $R > 0$ and we tile $C(z'_0,R_1)$ by squares of sidelength $\approx R$. We let $\{C_i\}_{i \in I}$ be this collection of squares and $x_i$ be the center of $C_i$. We sample a point configuration $\C$ in $C(z'_0,R_1)$ according to the law $\B_{\Nint}$ of a Bernoulli point process with $\Nint$ points, and we decompose $\C$ as $\C = \sum_{i \in I} \C_i$ where $\C_i := \C \cap C_i$ is the point configuration in $C_i$.
We form two points processes, the continuous average $\M_1 := \iNxb(\C)$ and the discrete average $\M_2 := \frac{1}{\# I} \sum_{i \in I} \delta_{\C_i}$.
Classical large deviations arguments (similar to Section \ref{sec:entropy}) show that both $\M_1$ and $\M_2$ bleong to $B(P, \epsilon)$ with probability $\approx \exp(-N^{2\deltap} \ERS[P|\Poisson^{\mueq(z)}])$.

Then we apply to each point configuration $\C_i$ the “screening-then-regularization” procedure of \cite[Section 5]{LebSer}. The screening procedure is similar in spirit to the one described in Lemma~\ref{lem:screening}, except that here we change $\C_i$ to $\Cscr_i$ by modifying the configuration only in a thin layer near $\partial C_i$ and we construct an electric field $\Escr_i$ compatible with $\Cscr_i$ and which is screened outside $C_i$ (whereas in Lemma \ref{lem:screening} we rather “screen the configurations from the inside”). By gluing the fields $\Escr_i$ together we define $\Escr$ which is compatible with $\Cscr := \sum_{i \in I} \Cscr_i$. The next task is to “regularize” the point configurations, which means to separate the pairs of points which are too close from each other. This changes $\Cscr$ into $\Cmod$ (which is very much like $\Cscr$) and $\Escr$ into an electric field $\Emod$ which is still screened outside $C(z'_0,R_1)$. 

The energy of $\Escr$ can be seen to satisfy, for any $0 < \eta < 1$
\[
\int_{C(z'_0,R_1)} |\Escr_{\eta}|^2 + \Nint \log \eta = \sum_{i \in I} \int_{C_i} |(\Escr_{i})_{\eta}|^2 + \Nint \log \eta,
\]
and a certain continuity property of the energy shows that the right-hand side is smaller than $N^{2\deltap}(\W_{\meq(z_0)}(P) + o(1))$ often enough. Passing from $\Escr$ to $\Emod$ does not affect this estimate, on the contrary the regularization procedure allows to bound the difference between the truncated energy $\int_{C(z'_0,R_1)} |\Emod_{\eta}|^2 + \Nint \log \eta$ and its limit as $\eta \t0$. 
\end{proof}

\subsection{Choice of the deltas}
In the rest of the proof, given $0 <\delta \leq 1/2$ we will need to fix $\deltatp, \deltapp, \deltap$ satifying some inequalities. 
\begin{lem}  \label{lem:choixdelta} Let $\gamma := \sqrt{\frac{1 + \kappa/2}{1+\kappa/3}}$, with $\kappa$ as in \eqref{assum-Holder}. Since $0 < \kappa \leq 1$ we have $1 < \gamma  \leq \frac{3\sqrt{2}}{4} \approx 1.06$. Let also $\alpha := \frac{\gamma - 1}{1 - \frac{\gamma}{3}}$, we have $\alpha \in (0,1)$. 

For any $0 <\delta \leq 1/2$ and any $\deltap, \deltapp, \deltatp$ such that
\begin{equation} \label{choixdelta}
\delta > \deltap > \max \left(\frac{3}{4} \delta, \delta \frac{1-\alpha}{1-\alpha^2}, \delta(1+\kappa/2) - \kappa/2 \right), \deltatp = \frac{1}{3} \delta \gamma, \quad \deltapp = \alpha^2 \deltatp +  (1-\alpha^2) \deltap,
\end{equation}
we have $0 < \deltatp < \deltapp < \deltap < \delta$, $\deltap > \frac{2}{3} \delta$,  $3\deltatp > \delta$,  $\deltap + 3 \deltatp + \kappa(\deltatp - 1/2) < 2 \deltap$, $2\delta < \deltapp + 3 \deltatp$ and $2\delta - \deltapp < 2 \deltap$. Moreover, if we consider the lower bound on $\deltap$ as a function $f(\delta)$, we have $f^{\circ k}(\delta) \t0$ as $k \ti$, where $f^{\circ k}$ denotes the $k$-th iteration of $f$.
\end{lem}
\begin{proof}
It is clear that $\deltatp > 0$. From the fact that $\deltap \geq \delta(1+\kappa/2)-\kappa/2$ and the expression of $\deltatp$ we get $\deltap + 3 \deltatp + \kappa(\deltatp - 1/2) < 2 \deltap$. Since $\kappa > 0$ we have $3 \deltatp > \delta$. Since $\kappa \leq 1$ we also have $\deltatp \leq \frac{1}{3} \sqrt{\frac{9}{8}} \delta < \frac{3}{4} \delta$ hence $\deltatp < \deltap$. On the other hand from the definition of $\deltap$ it is clear that $\frac{2}{3} \delta < \deltap < \delta$ because $\kappa > 0$ and $\delta \leq 1/2$. Since $\deltap > \delta \frac{1-\alpha}{1-\alpha^2}$ and since $\alpha < \alpha^2$, the inegality $2\delta < \deltapp + 3 \deltatp$ follows from checking that $\delta(2 - (1+\alpha/3)\gamma) \leq (1-\alpha)\delta$. The inequality $2 \delta < \deltapp + 3 \deltatp$ implies the last one, $2\delta - \deltapp < 2 \deltap$, because it can be easily checked that $3 \deltatp = \delta \gamma < \frac{6}{4} \delta < 2\deltap$.

Finally, we may observe that $f(\delta) \leq \max(\frac{3}{4} \delta, \frac{1-\alpha}{1-\alpha^2} \delta, \delta - \kappa/4)$ hence $f^{\circ k}(\delta)$ is decreasing and tends to $0$ as $k \ti$.
\end{proof}

\subsection{A LDP lower bound}
We use Lemma \ref{lem:generermicro} and the screening result of Lemma \ref{lem:screening} to prove a first LDP lower bound.
\begin{prop} \label{prop:LDPLB1}
Let $0 < \delta \leq 1/2$ and $z_0 \in \mathring{\Sigma}$ be fixed. Assume that a good control holds at scale $\delta$ and let us fix $\deltap, \deltapp, \deltatp$ as in \eqref{choixdelta}. For any $P \in \probas_{s, \meq(z_0)}(\config)$ we have
\begin{multline} \label{LDPLB}
\log \PNbetaxb \left(B(P,\epsilon)\right) \geq - N^{2\deltap} \fbeta^{\meq(z_0)}(P) - \log \KNbeta \\ + \max_{R_1, R_2, \Nou} \log  {N \choose \Nou} \left(\int e^{-\beta (\Fou(\Xou) + N \Zeta(\Xou))} d\Xou\right)  + o(N^{2\deltap}).
\end{multline}
\end{prop}
The maximum $\max_{R_1, R_2, \Nou}$ is taken among $\{R_1, R_2\}$ satisfying \eqref{R2bon} and \eqref{R1bon} and with $\Nou$ between $1$ and $N$.
\begin{proof}
By definition of $\PNbetaxb$ and $\PNbeta$ it is enough to prove
\begin{multline} \label{LDPLB2}
\log \int_{(\iNxb)^{-1} (B(P, \epsilon))} e^{-\beta (\wN(\nu'_N) + N\Zeta(\nu_N))} d\XN  \geq - N^{2\deltap} \fbeta^{\meq(z_0)}(P) \\ + \max_{R_1, R_2, \Nou} \log  {N \choose \Nou} \left(\int e^{-\beta (\Fou(\Xou) + N \Zeta(\Xou))} d\Xou\right) + o(N^{2\deltap}).
\end{multline}

Let $R_1, R_2, \Nou$ be fixed. Let $\Xou$ be a finite point configuration in $C(z'_0, R_2)^c$ such that $\Fou(\Xou)$ is finite. Let $E \in \Elecou(\Xou)$ be a minimizer in the definition of $\Fou$.  
We claim that there exists a set $\mathcal{A}^{\rm{tot}}$ of $N$-tuples $\XN$ such that $\XN = \Xou$ on $C(z'_0, R_2)^c$, that the energy is controlled uniformly on $\Atot$ as follows
\[
\wN(\nu'_N)  \leq - (\W_{\meq(z_0)}(P) + \Fou(\Xou)) + o(N^{2\deltap}),
\]
and that the volume of $\Atot$ is almost optimal
\[
\log \Leb^{\otimes N}(\Atot) \geq - \ERS[P|\Poisson^{\meq(z_0)}] + \log {N \choose \Nou } + o(N^{2\deltap}).
\]
\textbf{Step 1.} \textit{Screening $E$.} We may apply Lemma \ref{lem:screening} to the point configuration $\Xou$ and the electric field $E$, with $\mu = \mupeq$. Let us check that the assumptions of Lemma \ref{lem:screening} are satisfied:
\begin{enumerate}
\item The first condition on $R_1, R_2$ is satisfied by assumption (see \eqref{R2bon} and \eqref{R1bon}).
\item Since $\Fou(\Xou)$ is finite, \eqref{nointersection} holds i.e. the smeared out charges at scale $\eta_1$ do not intersect $\partial C(z'_0,R_2)$.
\item The third and fourth condition on $\Nint, \Nmid$ follow from the fact that $R_1$ is a good interior boundary.
\end{enumerate}
From Assumption \ref{assumption:V} we know that $\mupeq \preceq 1$, and since we are blowing-up the configuration around $z_0  \in \mathring{\Sigma}$ the density $\mpeq$ is bounded below on $C(z'_0, N^{\delta})$ by some $\um > 0$ depending on $z_0$. We deduce from \eqref{assum-Holder} that $|\mupeq(x) - \mupeq(y)| \preceq N^{-\kappa/2} |x-y|^{\kappa}$, hence we may chose $\Chol = N^{-\kappa/2}$ in \eqref{assum:muHolder}.

By definition of a good exterior boundary (see \eqref{bonbord1}) we have 
$\int_{\partial C(z'_0, R_2)} |E_{\eta_1}|^2 \preceq N^{2\delta-\deltapp}\log^2N$, thus \eqref{scrineg} is satisfied (for $N$ large enough) as long as $2\delta < \deltapp + 3\deltatp$ (which is ensured by the choice \eqref{choixdelta}).

We obtain a family $\AtranN$ of point configurations such that the conclusions of Lemma \ref{lem:screening} hold.
\\

\textbf{Step 2.} \textit{Generating microstates.}
Now we apply Lemma \ref{lem:generermicro} with $R_1$ as above and obtain a family $\AintN$ of point configurations with $\Nint$ points in $C(z'_0, R_1)$ together with screened electric fields $\Eint$ such that \eqref{presdeP}, \eqref{bonnenergie1}, \eqref{bonvolume1} are satisfied.
\\

\textbf{Step 3.} \textit{Gluing pieces together and bounding the energy.}
For any $\Ctran \in \AtranN$ and $\Cint \in \AintN$ we form the configuration 
\[
\Ctot := \Ctran + \Cint + \Xou.
\]
It is easy to check that $\Ctot$ always has $N$ points. Indeed we know that
\begin{itemize}
\item By Lemma \ref{lem:generermicro}, $\Cint$ always has $\int_{C(z'_0, R_1)} d\mupeq$ points. 
\item By construction, $\Ctran$ has $\Ntran = \int_{\partial C(z'_0, R_2)} E_{\eta_1} \cdot \vec{n} - \int_{C(z'_0, R_2) \backslash C(z'_0, R_1)} d\mupeq$ points.
\item By integrating the compatibility relation of $E$ and $\Xou$, we get \[\Nou = N- \int_{\partial C(z'_0, R_2)} E_{\eta_1} \cdot \vec{n} + \int_{C(z'_0, R_2)} d\mupeq.\]
\end{itemize}

If $\Etran$ and $\Eint$ are the electric fields associated to $\Ctran$ and $\Cint$ we also define 
\[
\Etot := \Etran + \Eint + E \mathbf{1}_{C(z'_0, R_2)^c}
\]
 By construction the normal derivatives of $\Etran$ and $\Eint$ coincide on $\partial C(z'_0,R_1)$ (they both vanish), and the normal derivatives of $\Etran_{\eta}$ and $E_{\eta}$ coincide on $\partial C(z'_0,R_2)$ for any $\eta \leq \eta_{1}$ (they coincide for $\eta_1$ by construction, but since there are no points at distance $\leq \eta_{1}$ of $\partial C(z'_0,R_2)$ the value of the fields $E_{\eta}$ and $\Etran_{\eta}$ on $\partial C(z'_0,R_2)$ do not depend on $\eta$ for $\eta \leq \eta_1$). Thus $\Etot$ satisfies
\begin{enumerate}
\item $- \div \Etot = \cds( \C^{\mathrm{tot}} - \mupeq) \text{ in } \R^2$
\item $\Etot$ coincides with $E$ on $C(z'_0,R_2)^c$. In particular $\Etot$ belongs to $\Eleco$, as $E$ does.
\item The energy of $\Etot$ is bounded as follows
\begin{equation} \label{controle-Etot}
\lim_{\eta \t0} \left( \int_{\R^2} |\Etot_{\eta}|^2 + N \log \eta \right) \leq \Fou(\Xou) + N^{2\deltap} \bW_{\meq(z_0)}(P) + o(N^{2\deltap}) \text{ as } N \ti.
\end{equation}
\end{enumerate}
To show \eqref{controle-Etot}, let us split the energy of $\Etot$ as 
\begin{multline} \label{decompEtot}
 \int_{\R^2} |\Etot_{\eta}|^2 + N \log \eta = \left(\int_{C(z'_0,R_2)^c} |E_{\eta}|^2 + \Nou \log \eta\right) \\ + \left(\int_{C(z'_0,R_2) \backslash C(z'_0,R_1)} |\Etran_{\eta}|^2+ \Ntran \log \eta \right)  + \left(\int_{C(z'_0,R_1)} |\Eint_{\eta}|^2 + \Nint \log \eta\right)
\end{multline}
By definition of $\Fou(\Xou)$ and by the choice of $E$ we have
\begin{equation} \label{energieext}
\frac{1}{2\pi} \lim_{\eta \t0} \left(\int_{C(z'_0,R_2)^c} |E_{\eta}|^2 + \Nou \log \eta\right) = \Fou(\Xou).
\end{equation}
In view of \eqref{bonnenergie1} we have
\begin{equation} \label{energieint}
\frac{1}{2\pi} \lim_{\eta \t0} \left(\int_{C(z'_0,R_1)} |\Eint_{\eta}|^2 + \Nint \log \eta \right) \leq N^{2\deltap} \bW_{\meq(z_0)}(P) + o(N^{2\deltap}).
\end{equation} 
Finally, the conclusions of Lemma \ref{lem:screening} combined with the control \eqref{bonbord1} and the fact that $\Chol = N^{-\kappa/2}$ ensure that
\begin{multline*}
\lim_{\eta \t0} \left(\int_{C(z'_0,R_2) \backslash C(z'_0,R_1)} |\Etran_{\eta}|^2+ \Ntran \log \eta \right)  \preceq  N^{\deltatp} N^{2\delta - \deltapp}(\log N)^2 + N^{\deltap + 3\deltatp} N^{\kappa (\deltatp-1/2)} \\ + N^{\deltap + \deltatp} \log N.
\end{multline*}
The choice of $\deltap, \deltapp, \deltatp$ as in \eqref{choixdelta} yields
\begin{equation} \label{energietran}
\lim_{\eta \t0} \left(\int_{C(z'_0,R_2) \backslash C(z'_0,R_1)} |\Etran_{\eta}|^2+ \Ntran \log \eta \right) \ll N^{2\deltap}.
\end{equation}
Inserting \eqref{energieext}, \eqref{energieint} and \eqref{energietran} into \eqref{decompEtot} yields \eqref{controle-Etot}.

Now, using the minimality of local energy as stated in Lemma \ref{lem:minilocale} and the formula \eqref{REfinite} we conclude that
\begin{multline} \label{conseqwNctot}
\wN(\Ctot) \leq  \frac{1}{2\pi} \lim_{\eta \t0} \left( \int_{\R^2} |\Etot_{\eta}|^2 + N \log \eta\right) \leq \Fou(\Xou) + N^{2\deltap} \bW_{\meq(z_0)}(P) \\ + o(N^{2\deltap}) \text{ as } N \ti.
\end{multline}

\textbf{Step 4.} \textit{Volume considerations.}
For any $\Xou$ we let $\Atot(\Xou)$ be the set of point configurations $\Ctot$ obtained as above. Now, let $\cA$ be a measurable set of finite point configurations $\Xou$ with $\Nou$ points in $C(z'_0,R_2)^c$ such that $\Fou(\Xou)$ is finite for all $\Xou \in \cA$. We let $\cAtot$ be 
\[
\cAtot := \bigcup_{\Xou \in \cA} \Atot(\Xou).
\]
Using the volume estimate \eqref{AextNvol} we obtain, with the choice \eqref{choixdelta}
\begin{equation*} 
\log \Leb^{\otimes N}(\cAtot) \geq \log \Leb^{\otimes \Nou}(\cA) + \log \Leb^{\otimes \Nint}(\AintN) + o(N^{2\deltap}) + \log { N \choose \Nou \Ntran \Nint},
\end{equation*}
(where the last term denotes a multinomial coefficient).
Using \eqref{bonvolume1} and a straightforward combinatorial inequality yields
\begin{equation} \label{volumeAtot}
\log \Leb^{\otimes N}(\cAtot) \geq \log \Leb^{\otimes \Nou}(\cA) - N^{2\deltap}\Ent[P|\Poisson^{\meq(z_0)}] + o(N^{2\deltap}) +\log {N \choose \Nou}.
\end{equation}
This proves the claim made before Step 1.
\\
\textbf{Step 5.} \textit{Conclusion.} Combining \eqref{conseqwNctot} and \eqref{volumeAtot}, we obtain that
\begin{multline*}
\log \int_{(\iNxb)^{-1} (B(P, \epsilon))} e^{-\hal \beta (\wN(\nu'_N) + N\Zeta(\nu_N))} d\XN  \\ \geq \log {N \choose \Nou} \left(\int e^{-\hal \beta (\Fou(\Xou) + N \Zeta(\Xou))} d\Xou\right) - N^{2\deltap}(\Ent[P|\Poisson^{\meq(z_0)}]  + \hal \beta \bW_{\meq(z_0)}(P)) \\ + o(N^{2\deltap}),
\end{multline*}
for any choice of $R_1, R_2$ and $\Nou$ as in the definitions \ref{def:extboundary} and \ref{def:intboundary}. It yields \eqref{LDPLB2}.
\end{proof}

\section{Conclusion} \label{sec:conclusion}
\subsection{Proof of Theorem \ref{theo:main}}
\subsubsection{Good control at macroscopic scale}
\begin{lem}
\label{lem:GC12} Good control holds at scale $\delta = \frac{1}{2}$.
\end{lem}
\begin{proof} Let $z_0 \in \mathring{\Sigma}$. 
Using \eqref{def:Gibbs2} we see that 
\[
\log \PNbeta(\wN(\XN) \geq NM) \leq - \log \KNbeta + \log \int e^{- \frac{\beta}{2} (NM + N \Zeta(\XN))} d\XN.
\]
Since $\wN$ is bounded below by $O(N)$ (see Lemma \ref{lem:REfinite}) we have $- \log \KNbeta + \log \left(\int e^{- \frac{\beta}{2} N \Zeta(\XN)}\right)  = O(N)$ (which may depend on $\beta$), hence we get
\[
\log \PNbeta\left(\wN(\XN) \geq M\right) \leq N \left(- \frac{\beta}{2} M + O(1)\right).
\]
We deduce that for $M_0$ large enough,  $\limsup_{N \ti} \frac{1}{N} \log \PNbeta(\wN(\XN) \geq M_0) < 0$,
which in particular implies that $\limsup_{N \ti} \frac{1}{N^{2\delta}} \log \PNbeta(\wN(\XN) \geq M_0)$ for any $\delta < \frac{1}{2}$, thus we have $\wN(\XN) \preceq N$ with $\delta$-overhelming probability. Using \eqref{REfinite}
and Lemma \ref{lem:monoton1} we get for any $\eta \in (0,1)$ that
\[
\lim_{\eta \t0} \int_{\R^2} |\Eloc_{\eta}|^2 + N \log \eta \preceq_{\delta} N, 
\]
which yields \eqref{goodcont1}, and we deduce \eqref{goodcont2} from the discrepancy estimates of Lemma \ref{lem:discr}.
\end{proof}

\subsubsection{Exponential tightness}
\begin{lem}
For any $0 < \delta  \leq 1/2$, if good control holds at scale $\delta$ then $\PNbetaxb$ is exponentially tight (at speed $N^{2\deltap}$) for any $z_0 \in \mathring{\Sigma}$ and $\frac{2}{3} \delta < \deltap < \delta$.
\end{lem}
\begin{proof}
Let $z_0 \in \mathring{\Sigma}$ and $\frac{2}{3} \delta < \deltap < \delta$ be fixed. The good control at scale $\delta$, combined with Lemma \ref{lem:discr3}, implies that there exists $C > 0$ such that the number of points in $C(z'_0, N^{\deltap})$ is bounded above by $C N^{2\deltap}$ with $\deltap$-overhelming probability. It implies that $\PNbetaxb$ is concentrated on the compact subset 
\[
\left\lbrace P \in \probas(\config), \Esp_{P} [\Nn(0,R)] \leq C R^2\quad  \forall R  > 0 \right\rbrace,
\]
with $\deltap$-overhelming probability, which ensures exponential tightness at speed $N^{2\deltap}$. 
\end{proof}
\subsubsection{Proof of the theorem}
\begin{proof}[Proof of Theorem \ref{theo:main}]
\textbf{Step 1.} \textit{Good control $\then$ LDP.}

In this first step we claim that if a good control holds at scale $\delta$, then the LDP of Theorem \ref{theo:main} holds for $\deltap$ as in \eqref{choixdelta}.

Indeed, comparing the right-hand side of \eqref{LDPLB} with the definition \eqref{def:KNbetax} of $\KNbetax$ we see that $\lim_{N \ti} N^{-2\deltap}  \log \KNbetax$
exists and that the following weak large deviation principle holds
\begin{equation} \label{weakLDP1}
\lim_{\epsilon \t0} \lim_{N \ti} N^{-2\deltap} \log \PNbetaxb \left(B(P,\epsilon)\right) = - \fbetax(P) - \lim_{N \ti} N^{-2\deltap}  \log \KNbetax,
\end{equation}
for any $P \in \probas_{s,\meq(z_0)}(\config)$, hence since $\bW_{\meq(z_0)}(P) = + \infty$ as soon as $P$ is not of intensity $\meq(z_0)$, we may write \eqref{weakLDP1} for any $P \in \probas_{s}(\config)$. By exponential tightness we obtain a full large deviation inequality: for any measurable  $A \subset \probas_{s}(\config)$ it holds
\begin{multline} \label{FLDP1}
- \inf_{P \in \mathring{A}} \fbeta^{\meq(z_0)}(P) - \lim_{N \ti} N^{-2\deltap} \log \KNbetax \\
\leq \liminf_{N \ti}  N^{-2\deltap} \log \PNbetaxb(A) \leq 
\limsup_{N \ti} N^{-2\deltap} \log \PNbetaxb(A) \\ \leq - \inf_{P \in \bar{A}} \fbetax(P) - \lim_{N \ti} N^{-2\deltap}  \log \KNbetax.
\end{multline}
In particular, taking $A = \probas_s(\config)$ we obtain
\begin{equation} \label{asymKNbetax}
\lim_{N \ti} N^{-2\deltap} \log \KNbetax = \inf \fbetax(P),
\end{equation}
and inserting \eqref{asymKNbetax} into \eqref{FLDP1} yields the LDP for $\{\PNbetaxb\}_{N}$ as stated in Theorem \ref{theo:main}.

\textbf{Step 2.} \textit{Good control $\then$ good control.}

We now claim that if a good control holds at scale $\delta$, then it holds at scale $\deltap$ with $\deltap$ as in \eqref{choixdelta}.
Combining the “good control upper bound” of Lemma \ref{lem:GCUB} and the lower bound estimates which yield \eqref{asymKNbetax} we deduce that
\begin{equation} \label{GCUB2}
\log \PNbetaxb \left( \mc{E}_{M} \right) \preceq  -MN^{2\deltap} + O(N^{2\deltap}),
\end{equation}
where $\mc{E}_M$ is as in \eqref{def:mcEM}. 
In particular it implies that $\int_{C(z'_0, N^{\deltap})} |\Eloc_{\eta_0}|^2 + \Nxa \log \eta_0 \preceq_{\delta'} N^{2\deltap}$ for any $\delta' < \deltap$ and any $\eta_0 \in (0,1)$. We also have $\Nxa \preceq_{\deltap} N^{2\deltap}$ (since it was proven in Lemma \ref{lem:intermedUB} that \eqref{ninmnxa} holds with $\deltap$-overhelming probability) hence in particular $\Nxa \preceq_{\delta'} N^{2\deltap}$ for $\delta' < \deltap$.

\textbf{Step 3.} \textit{Conclusion.}

Combining both steps with the initialization of Lemma \ref{lem:GC12} and the conclusions of Lemma \ref{lem:choixdelta} yields the proof of Theorem \ref{theo:main}.
\end{proof}
\subsection{Proof of Corollary \ref{coro:discr}}
\begin{proof}
We simply combine the fact that a good control holds at any scale $0 < \delta \leq 1/2$ (which follows from Theorem \ref{theo:main}) with Lemma \ref{lem:discr2}.
\end{proof}

\subsection{Proof of Corollary \ref{coro:locallaw}}
\begin{proof}
We may split $C(z'_0, N^{\delta})$ into a family $\{C_i\}_{i \in I}$ of squares of sidelength $\approx N^{\deltap}$, with $\# I \approx N^{2(\delta-\deltap)}$. For any $i \in I$ we have, letting $z_i$ be the center of $C_i$ and $\D_i$ the discrepancy in $C_i$
\[
\int_{C_i} f (d\nu'_N - d\mupeq) = \D_i f(z_i) + \int_{C_i} (f(z) - f(z_i)) (d\nu'_N - d\mupeq).
\]
Since good control holds at scale $\delta$ we have $|\D_i| \preceq_{\deltap} N^{4\delta/3}$ (from the discrepancy estimates of Lemma \ref{lem:discr2}) and $\int_ {C(z'_0, N^{\delta})} d\nu'_N \preceq _{\deltap} N^{2\delta}$. On the other hand $\int_{C(z'_0, N^{\delta})} d\mupeq \preceq N^{2\delta}$ (because $\mupeq$ is bounded above). Moreover the mean value theorem yields $|f(z) - f(z_i)| \leq N^{\deltap} ||\nabla f||_{\infty}$. We thus have
\[
\left|\sum_{i \in I} \int_{C_i} f (d\nu'_N - d\mupeq) \right| \preceq N^{2(\delta-\deltap)} N^{4\delta/3} ||f||_{\infty} + N^{2\delta} N^{\deltap} ||\nabla f||_{\infty},
\]
hence we see that
\[
N^{-2\delta} \left| \int_{C(z'_0, N^{\delta})} f (d\nu'_n - d\mupeq) \right| \preceq_{\deltap} ||\nabla f||_{\infty} N^{\deltap} + ||f||_{\infty} N^{-2\delta/3},
\]
which concludes the proof since $\deltap < \delta$.
\end{proof}

\section{Additional proofs} \label{sec:annexe}
\subsection{Proof of Lemma \ref{lem:monoton2}} \label{sec:preuvemonot}
\begin{proof}
We may decompose $\Eloc$ as $\Ein + \Eou$ where $\Ein$ is the local electric field generated by the electric system \textit{inside} $C_{R_2}$ and $\Eou$ is the local electric field generated by the electric system \textit{outside} $C_{R_2}$. We have
\begin{equation*}
\int_{C_{R_2}} |\Eloc_{\eta}|^2 = \int_{C_{R_2} } |\Ein_{\eta}|^2 + \int_{C_{R_2} } |\Eou_{\eta}|^2 + 2 \int_{C_{R_2}} \Ein_{\eta} \cdot \Eou_{\eta}.
 \end{equation*}
Since the charges outside $C_{R_2}$ are at distance at least $\eta_1$ from $\partial C_{R_2}$ we may replace $\Eou_{\eta}$ by $\Eou$ in the previous identity (in fact we have $\Eou_{\eta} = \Eou$ on $C_R$ for $\eta \leq \eta_1$). Integrating by parts we obtain
\begin{equation*}
 \int_{C_{R_2}} |\Eloc_{\eta}|^2 = \int_{C_{R_2}} - \Hin_{\eta} \Delta \Hin_{\eta} + \int_{\partial C_{R_2} } \Hin_{\eta} E_{\eta} \cdot \vec{n}
 + \int_{\partial C_{R_2} } \Hin_{\eta} \Eou \cdot \vec{n}
\end{equation*}
(up to additive terms which do not depend on $\eta \leq \eta_1$), where $\Hin$ is the local electric potential generated by the electric system inside $C_{R_2}$. By assumption we have $\Hin_{\eta} = \Hin_{\eta_1}$ and $\Ein_{\eta} = \Ein_{\eta_1}$ on $\partial C_{R_2}$. Finally we see that
\[
\int_{C_{R_2}} |\Eloc_{\eta}|^2 - |\Eloc_{\eta_1}|^2  = - \int_{\R^2} \left(\Hin_{\eta} \Delta \Hin_{\eta} -\Hin_{\eta_1} \Delta \Hin_{\eta_1}\right) 
 \]
for any $\eta \leq \eta_1$, and \eqref{monoton2} is obtained as Lemma \ref{lem:monoton1} (\textit{cf.} the remark after the statement of Lemma \ref{lem:monoton1}).
\end{proof}

\subsection{Proof of Lemma \ref{lem:minilocale}}
\begin{proof}
The neutrality of the system implies that the local electric potential $\Hloc$ decays like $|x|^{-1}$ as $|x|\to \infty$ in $\R^{2}$ and $\Eloc$ decreases like $|x|^{-2}$. If the right-hand side of \eqref{comparloc} is infinite then there is nothing to prove. If it is finite, given $M>1$ and letting $\chi_M$ be a smooth nonnegative function equal to $1$ in $\carr_M$ and $0$ at distance $\ge 1$ from $\carr_M$,  we may write
\begin{multline}
\int_{\R^{2}} \chi_M  |E_{\eta}|^2 = \int_{\R^{2}}\chi_M |E_{\eta} - \Eloc_{\eta}|^2 + \int_{\R^{2}}\chi_M  |\Eloc_{\eta}|^2 + 2 \int_{\R^{2}}\chi_M   (E_{\eta} - \Eloc_{\eta}) \cdot \Eloc_{\eta}
\\ \geq 
 \int_{\R^{2}}\chi_M  |\Eloc_{\eta}|^2 + 2 \int_{\R^{2}}\chi_M  (E_{\eta} - \Eloc_{\eta}) \cdot (\nabla \Hloc_{\eta}) \\ =  \int_{\R^{2}}\chi_M  |\Eloc_{\eta}|^2 + 2 \int_{\R^{2}}    \Hloc_{\eta} (E_{\eta} - \Eloc_{\eta}) \cdot \nab \chi_M,
\end{multline} where we have integrated by parts and we have used the fact that $E, \Eloc$ are compatible with the same configuration (hence $\div(E_{\eta} - \Eloc_{\eta}) = 0$). Letting $M\to \infty$, the last term tends to $0$ by finiteness of the right-hand side of \eqref{comparloc},the decay properties of $\Hloc$ and $\Eloc$ and the decay assumption on $E$.
\end{proof}

\subsection{Auxiliary estimate for screening}
\label{sec:annexe-screening}
\begin{lem} \label{lem:screenestim} Let $l > 0$ and let $H$ be a rectangle of $\R^2$ with sidelengths in $(l/2, 3l/2)$. Let $g \in L^2(\partial H)$ and let $m$ be a function on $H$ of average $m_0 := \frac{1}{|H|} \int_{H} m$ such that
\begin{equation} \label{compatibilitegm}
- \cds m_0 |H| = \int_{\partial H} g.
\end{equation}
Then there exists a solution $h$ to $- \Delta h = \cds m$ in $H$ with $\nabla h \cdot \vec{n} = 0$ on  $\partial H$ satisfying
\begin{equation} \label{correctingflux}
\int_{H} |\nabla h|^2 \preceq l \int_{\partial H} |g|^2 + l^4 ||m-m_0||^2_{L^{\infty}(H)}.
\end{equation}
\end{lem}
\begin{proof}
A solution exists thanks to the compatibility condition \eqref{compatibilitegm}. We may split $h$ as $h_1 + h_2$ where $h_1$ solves
\[
- \Delta h_1 = \cds m_0 \text{ in } H, \quad \nabla h_1 \cdot \vec{n} = g \text{ on } \partial H,
\]
and $h_2$ is the mean zero solution to
\[
- \Delta h_2 = \cds (m-m_0) \text{ in } H, \quad \nabla h_2 \cdot \vec{n} = 0 \text{ on } \partial H.
\]
In view of \cite[Lemma 5.8]{RougSer} we may find $h_1$ satisfying
\begin{equation} \label{estimh1}
\int_{H} |\nabla h_1|^2 \preceq l  \int_{\partial H} |g|^2.
\end{equation}
We also claim that 
\begin{equation} \label{estimh2}
\int_H |\nabla h_2|^2 \preceq l^4 ||m-m_0||^2_{L^{\infty}(H)}.
\end{equation}
Indeed it is easy to check that \eqref{estimh2} holds when $l=1$, and the general case follows by a scaling argument.

Combining \eqref{estimh1} and \eqref{estimh2} concludes the proof.
\end{proof}

\bibliographystyle{alpha}
\bibliography{loiloc}

\end{document}